\documentclass[11pt,a4paper]{article}
\usepackage[T1]{fontenc}
\usepackage{lmodern}
\usepackage[margin=27mm,headheight=14pt]{geometry}
\usepackage{amsmath,amssymb,amsthm,mathtools}
\usepackage{microtype,booktabs,longtable,array}
\usepackage{enumitem,listings,xurl}
\usepackage[hidelinks,pdfusetitle]{hyperref}
\usepackage{fancyhdr}
\usepackage{needspace}
\usepackage{float}
\usepackage{graphicx}
\allowdisplaybreaks[2]
\numberwithin{equation}{section}
\newtheorem{theorem}{Theorem}[section]
\newtheorem{lemma}[theorem]{Lemma}
\newtheorem{proposition}[theorem]{Proposition}
\newtheorem{corollary}[theorem]{Corollary}
\theoremstyle{definition}

\theoremstyle{remark}
\newtheorem{remark}[theorem]{Remark}

\newcommand{\nn}{\mathfrak n}
\newcommand{\C}{\mathbb C}
\newcommand{\Z}{\mathbb Z}

\newcommand{\ch}{\operatorname{ch}}
\newcommand{\mult}{\operatorname{mult}}
\newcommand{\Tor}{\operatorname{Tor}}
\newcommand{\height}{\operatorname{ht}}
\newcommand{\len}{\operatorname{len}}
\newcommand{\Aconst}{\mathcal A}

\newcommand{\code}[1]{\nolinkurl{#1}}
\title{A proof of Frenkel's bound for the hyperbolic Kac-Moody Lie Algebra
\texorpdfstring{$A_1^{++}$}{A1++}\\[0.3em]
\large Computer-assisted}
\author{Gabriel B. Legros}
\date{September 12, 2026}
\makeatletter
\newcommand{\outlinebreak}{\par\@nobreakfalse\everypar{}\penalty0}
\makeatother
\begin{document}

\begin{titlepage}
\maketitle
\thispagestyle{empty}
\begin{abstract}
We give a computer-assisted proof of Frenkel's root-multiplicity
bound for the rank-three hyperbolic Kac--Moody Lie algebra
$A_1^{++}$. For every root $\alpha$, we prove
$\dim g_\alpha\le p(1-(\alpha,\alpha)/2)$, where $p$ is the
ordinary partition function.
The proof combines exact affine characters with a coefficientwise
logarithmic majorant derived from parabolic homology. After Weyl reduction, analytic estimates establish the inequality on infinite
regions of large depth or large level. The remaining finite region
is verified by computer. Together with the known low-level formulas, the assembly of these estimates establish the conjecture for every root
of $A_1^{++}$.
\end{abstract}

\paragraph*{AI disclosure}
Generative AI, namely OpenAI's GPT6 Astra, was used to a great extent to develop critical ingredients of the proof. This assistance was
instrumental in making this proof feasible.  GPT6 Astra has also acted as editor to catch mistakes in this manuscript and has generated TikZ diagrams.  The proof was verified in Lean from foundational axioms up to the reduction in Section \ref{sec:remainder}.  The latter is done in Python for performance reasons.

\paragraph*{}
\textit{We dedicate this proof to Igor Frenkel, for stating such a beautiful conjecture.}
\vfill
\end{titlepage}

\tableofcontents
\clearpage

\section{Introduction}\label{sec:main}
Hyperbolic Kac--Moody Lie algebras form a natural class beyond finite
and affine type. Unlike the other two however, they have been very difficult to study. Exact multiplicity formulas and recursive
algorithms exist; the difficulty is to extract effective, uniform
information from them \cite{BM,Kang94root,KK}. They also appear in physics.  Indeed,
Vertex operators also realize these algebras
inside Lie algebras of physical string states, linking the multiplicity
problem with the counting of string excitations \cite{F85,Borch86,GN95}.

Let $p_k(n)$ denote the number of partitions of $n$ in $k$ colours,
defined by
\[
 \sum_{n\geq0}p_k(n)q^n=\prod_{j\geq1}(1-q^j)^{-k},
 \qquad p_k(0)=1,\qquad p_k(n)=0\quad(n<0).
\]
Frenkel proposed that, for a symmetric hyperbolic Kac--Moody algebra
of rank $r$, every root $\alpha$ should satisfy
\[
 \mult(\alpha)\leq p_{r-2}\!\left(1-\frac{(\alpha,\alpha)}2\right).
\]
The conjecture appears explicitly in Section 4, equation (4.19),
at the Chicago Summer Seminar on
\emph{Applications of Group Theory in Physics and Mathematical
Physics}, held in July 1982; the proceedings were published in
1985 \cite{F85}.

The physical motivation comes from the no-ghost theorem of Goddard
and Thorn and of Brower \cite{GT,Brower}. In the critical bosonic
string dimension $26$, physical states of nonzero momentum, after
quotienting by null states, can be represented by transverse
oscillators. At excitation number $N$, their dimension is $p_{24}(N)$.
Frenkel used this theorem to obtain root-multiplicity bounds in the
corresponding dimension-26 Lorentzian setting
\cite[Section 4]{F85}; see also \cite[Section 5.3]{Legros}.
In the vertex construction the physical-state condition is
$N=1-(\alpha,\alpha)/2$, explaining the argument of the partition
function.  Below rank $26$, longitudinal states can survive
and occur in the Kac--Moody algebra itself \cite{GN95,CKMN}.

The algebra studied here, and perhaps the simplest hyperbolic Kac-Moody Lie algebra, is the rank-three hyperbolic algebra with
Cartan matrix
\begin{equation}\label{eq:cartan}
 A=\begin{pmatrix}2&-1&0\\-1&2&-2\\0&-2&2\end{pmatrix},
 \qquad (\alpha_i,\alpha_j)=A_{ij}
\end{equation}
where $i,j\in\{-1,0,1\}$. We denote the associated complex
Kac--Moody Lie algebra by $g(A)$; it is also known as $A_1^{++}$,
$HA_1^{(1)}$, $AE_3$, or $\mathfrak F$ \cite{FF,CKMN}.
Its simple roots are $\alpha_{-1},\alpha_0,\alpha_1$, and the nodes
$0,1$ generate the affine subalgebra $A_1^{(1)}$.

Denote by $g_\alpha$ the root space of $g(A)$ at root $\alpha$.  Any positive root may be written as
\begin{equation}\label{eq:coordinates}
 \alpha=\ell\alpha_{-1}+d\alpha_0+(d+s)\alpha_1,
 \qquad m(\ell,d,s)=\dim g_\alpha.
\end{equation}

We call $\ell$ the level, $d$ the depth, and $s$ the charge. 

Feingold and Frenkel's 1983 paper \cite{FF} studied this algebra
through its affine subalgebra and its decomposition by level, and
related its structure to Siegel modular forms of genus two.
The basic affine representation gives the exact level-one
multiplicities, which attain Frenkel's bound. At level two they
identified an exterior square of the basic representation modulo the
submodule generated by a Serre relation and obtained an exact
multiplicity generating function. These low-level formulas and the
computed multiplicities supplied evidence for the rank-three
conjecture \cite{FF,Kac}; the character formulas are also developed
in \cite{BB,CKMN}.

The general conjecture is nevertheless false. Kac, Moody, and
Wakimoto's calculation for $E_{10}=E_8^{++}$, published in 1988
\cite{KMW}, gives a counterexample already at level two relative to
the affine $E_9$ subalgebra. There are roots of squared length $-6$
with
\[
 \mult_{E_{10}}(\alpha)=727>726=p_8(4)
 =p_8\!\left(1-\frac{(\alpha,\alpha)}2\right);
\]
see also the multiplicity table in \cite[Table 1]{BGN}.  The conjecture remained open for $A_n^{++}$.

Kang developed a homological theory of graded Lie algebras using Hochschild--Serre
spectral sequences \cite{Kang93spectral}, and applied it to the
root multiplicities of $HA_1^{(1)}$ in papers published in 1993 and
1994 \cite{Kang93,Kang94HA}. This program obtained formulas through
level five \cite[Introduction]{KK}; further multiplicity formulas using homology were given
in \cite{Kang94root,KK}. Related calculations for the family
$HA_n^{(1)}=A_n^{++}$ were carried out by Kang and Melville
\cite{KM94}. These works supply exact formulas, but evaluating them
and controlling their cancellations remain substantial
tasks.

Bauer and Bernard's 1996 preprint, published in 1997 \cite{BB},
combined the free-Lie presentation with the coset construction to
compute level-three multiplicities for $A_1^{++}$ and $E_{10}$.
Here the positive-level algebra is a \emph{quotient} of the free Lie
algebra on the level-one module by the ideal of relations, as in the
earlier constructions \cite{FF,KMW}. Their character calculation
takes logarithms of a Poincar\'e--Birkhoff--Witt product identity
\cite[Section 1.2, equations (7)--(17)]{BB}. This logarithmic
extraction is quite powerful.  Its earliest precedent can be found in Berman and Moody's 1979
multiplicity formula which uses the logarithm of the denominator identity
and M\"obius inversion \cite[Theorem 2]{BM}.
Kang and Kim subsequently developed generalized Witt and
denominator formulas for free and graded Lie algebras
\cite{KK96,KK}.

On the physics side, the string-theoretic approach uses the physical-state operators of
Del Giudice, Di Vecchia, and Fubini (DDF) \cite{DDF}.
Gebert and Nicolai exhibited both transverse and longitudinal DDF
states in $E_{10}$ \cite{GN95} and constructed affine vertex
operators at arbitrary level \cite{GN97}. Further work studied
multistring vertices, Sugawara operators, and the physical states
missing from the Kac--Moody subalgebra \cite{GNW96,GKN97,BGGN}.
For $A_1^{++}$ itself, Capolongo, Kleinschmidt, Malcha, and Nicolai
\cite{CKMN} developed explicit low-level constructions and analyzed
the affine and coset Virasoro structures. The longitudinal operators
and the removal of states imposed by the Lie algebra relations
complicate a description at higher levels.
The author spent many years pursuing this approach without success; related results on physical-state
Lie algebras and longitudinal vertex operators appear in the author's
2021 thesis \cite[Sections 5.3--5.4]{Legros}.

A different approach uses automorphic forms. Gritsenko and
Nikulin developed automorphic corrections of Lorentzian
Kac--Moody algebras \cite{GrNi97}; for the particular matrix
\eqref{eq:cartan}, their correction has a denominator given by the
Igusa modular form of weight $35$ \cite{GrNi96}.
Niemann constructed generalized Kac--Moody algebras with known
multiplicities and obtained upper bounds for $A_1^{++}$ from such
an ambient algebra \cite{Niemann}. Kim and Lee applied the
Hardy--Ramanujan--Rademacher method to the resulting modular-form
coefficients, obtaining asymptotic estimates for these bounds
\cite{KL13}. These comparison functions differ from $p(n)$, so
these bounds do not establish the sharper inequality
sought here.

A turning point came in September 2026, when AI assistance from OpenAI’s GPT6 Astra made significant contributions to the problem.

In this manuscript we prove the rank-three bound by combining exact
affine characters, a coefficientwise logarithmic majorant obtained
from parabolic homology, analytic estimates, and finite verification.
The logarithmic extraction and homological character identities
build on the methods above \cite{GL,BM,Kang94root,KK,BB}; the
minimal-resolution argument in Section~\ref{sec:homology} establishes
what is needed to use a finite homology truncation as an upper
bound.

\begin{theorem}[Frenkel's inequality for $A_1^{++}$]\label{thm:main}
For every root $\alpha$ of the algebra \eqref{eq:cartan},
\[
             \dim g_\alpha\le p\!\left(1-\frac{(\alpha,\alpha)}2\right).
\]
\end{theorem}
The proof is broken down in steps as follows.

\begin{proposition}\label{prop:architecture}
The following statements alongside Weyl reduction (Section \ref{sec:geometry}) together prove the theorem.
\begin{enumerate}[label=\textup{(\roman*)},leftmargin=2em]
\item (Section \ref{sec:low}) Real and null roots satisfy the bound. Level one roots satisfy $m(1,d,s)=p(d-s^2)$, and all level-two roots satisfy the bound. After Weyl reduction, level-two equality
occurs exactly for $0\le N\le19$.
\item (Section \ref{sec:analytical}) For $\ell\ge3$, $|s|\le\ell/2$, and $d\ge165\ell$, the bound is
strict.
\item (Section \ref{sec:homology}) For $\ell\ge1000$, $2\ell\le d\le165\ell$, and
$-\ell/2\le s\le0$, the bound is strict.
\item (Section \ref{sec:remainder}) The bound holds for the finite region
\begin{equation}\label{eq:bounded-region}
\begin{aligned}
 \mathcal B={}&\bigl\{(\ell,d,s)\in\Z^3:3\le\ell<1000,\quad
 2\ell\le d<165\ell,\\
 &\hspace{2em}-\lfloor\ell/2\rfloor\le s\le0\bigr\}.
\end{aligned}
\end{equation}
\end{enumerate}
\end{proposition}

The levels in parts (ii)--(iv) are those of the Weyl-reduced roots.

This proof is computer-assisted.  The reduction in Section \ref{sec:homology} involves a check by subdividing an infinite region into finitely many rectangles.  The finite region in \ref{sec:remainder} is a simple, but computational check.  Both of these would be difficult to do without a computer.

\begin{figure}[H]
\centering
\includegraphics[width=\linewidth]{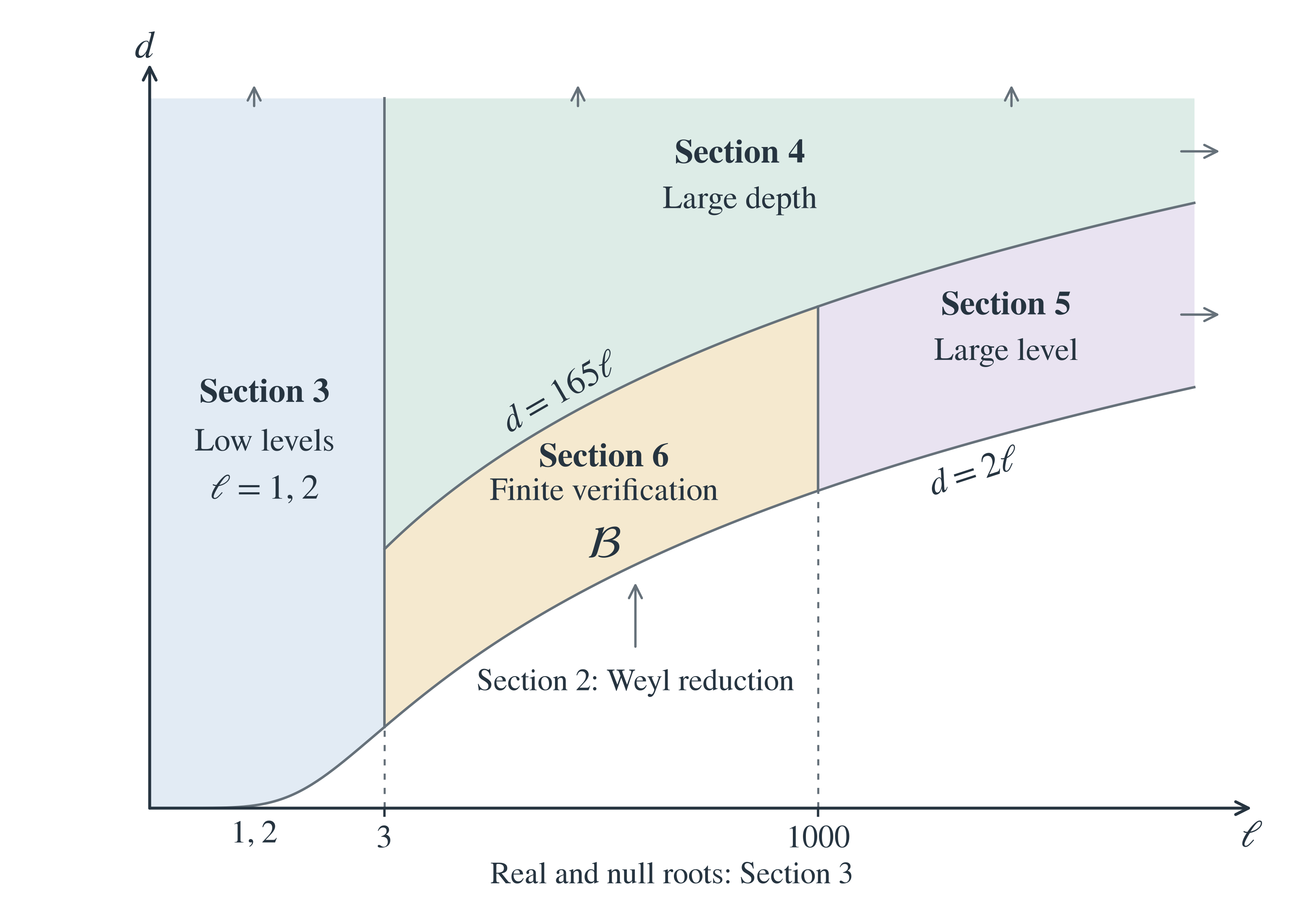}
\caption{The regions in Proposition~\ref{prop:architecture}, shown for
Weyl-reduced roots of negative norm. The charge satisfies
$-\ell/2\le s\le0$ throughout.}
\label{fig:proof-regions}
\end{figure}

\section{Root geometry, grading, and Weyl reduction}\label{sec:geometry}
The root geometry and affine-level structure of $A_1^{++}$ used below
were studied by Feingold and Frenkel \cite[Sections 2 and 4]{FF}.
We recall the required results in our conventions and give explicit
proofs of the reduction statements.

\subsection{Basic definitions and conventions}

The algebra $g(A)$ is graded by a root lattice, with root decomposition
$$g(A) = \mathfrak h \oplus \bigoplus\limits_{\alpha \in \Delta} g_\alpha$$
where $\mathfrak h$ is the Cartan subalgebra and $\Delta = \Delta^+ \cup \Delta^-$ is the set of roots, split into positive and negative roots.  A root is real if it is Weyl conjugate to a simple root, and imaginary otherwise.  For this algebra, all real roots satisfy $(\alpha, \alpha) = 2$, and imaginary roots satisfy $(\alpha, \alpha) \leq 0$ \cite[Section 2, Proposition 2.1]{FF}; see also \cite[Chapter 5]{Kac} for the general root-system theory.

Let
\begin{equation}\label{eq:N}
 (\alpha,\alpha)=2(\ell^2-\ell d+s^2),\qquad
 N(\ell,d,s)=1-\frac{(\alpha,\alpha)}2=\ell d-\ell^2+1-s^2
\end{equation}
and $p(n)$ be the ordinary partition function, with $p(0)=1$ and $p(n)=0$ for $n<0$.

The simple-root coefficients of positive roots are nonnegative and those of negative roots are nonpositive.  The Chevalley involution is an automorphism of $g(A)$ sending $g_\alpha$ to $g_{-\alpha}$.  Therefore, opposite roots have the same multiplicity, and to prove the theorem it is enough to prove it for positive roots.

For a vector space $M$ graded by the nonnegative integer span of the simple roots, with finite-dimensional graded pieces and $\ell > 0$, write
\begin{equation}\label{eq:character}
 \ch M=\sum_{\ell,d,s}\dim M_{\ell,d,s}\,t^\ell q^d z^s.
\end{equation}
These characters are formal series, completed by height $\height(\alpha)=\ell+2d+s$. Each degree has only finitely many decompositions into positive degrees, so products, and inverses and logarithms of series with constant term $1$, are well-defined.
For two characters $F$ and $G$, we write $F\succeq G$ if every
coefficient of $F-G$ is nonnegative.

\subsection{Full reduction of imaginary roots}
The Weyl group acts on the root lattice, preserving roots, root multiplicity and norm.  Also, simple reflections send positive roots to other positive roots, except for the corresponding simple root \cite[Chapters 3 and 5]{Kac}.  The Weyl action and root geometry for this algebra are developed in \cite[Section 2]{FF}. We use these to reduce the problem and acquire certain assumptions on level, depth and charge.

\begin{lemma}\label{lem:full-reduction}
Every positive imaginary root can be moved by the Weyl group to a
positive root satisfying
\begin{equation}\label{eq:chamber}
 d\ge2\ell,\qquad -\ell/2\le s\le0.
\end{equation}
After reduction, if its norm is negative, then $\ell>0$. If its norm is zero, then the
reduced root is $d(\alpha_0+\alpha_1)$, with coordinates $(0,d,0)$ and $d>0$.
\end{lemma}

\begin{proof}
Write $\alpha = \ell \alpha_{-1} + d \alpha_0 + (d+s) \alpha_1$.  Then
\begin{equation}\label{eq:pairings}
 (\alpha,\alpha_{-1})=2\ell-d,\qquad
 (\alpha,\alpha_0)=-\ell-2s,\qquad
 (\alpha,\alpha_1)=2s.
\end{equation}
The height of a root is the sum of its simple-root coefficients.  If $(\alpha,\alpha_i)>0$, the simple reflection $r_i(\alpha)=\alpha-(\alpha,\alpha_i)\alpha_i$ lowers the height by this positive integer.  The root remains positive and imaginary, so its height remains a positive integer.  The process therefore terminates.  At termination, $2\ell - d \leq 0$, $-\ell-2s \leq 0$ and $2s \leq 0$, proving the inequalities.

Now assume the root is reduced.  If $\ell = 0$ then $s = 0$ as well, hence $(\alpha, \alpha) = 0$. Thus a reduced root of negative norm has $\ell>0$.

Furthermore, if $(\alpha, \alpha) = 0$,
$$0 = (\alpha, \alpha) = 2(\ell^2 - \ell d + s^2) \leq - \frac{3}{2} \ell^2 \leq 0.$$
The inequalities force $\ell=0$.  Then $s=0$, and $d>0$ because the root is positive and nonzero.
\end{proof}

The following lemma will be used for the full examination of multiplicities at level 2.

\begin{lemma}[Charge reduction at level two]\label{lem:level2-charge-reduction}
Every level-two root with coordinates $(2,d,s)$ is affine Weyl
conjugate to a root $(2,d',s')$ with $s'\in\{0,1\}$, where
\[
 s'=\begin{cases}0,&s\text{ even},\\1,&s\text{ odd},\end{cases}
 \qquad d'=d-\frac{s^2-(s')^2}{2}.
\]
Consequently, we may study level 2 multiplicities by the cases
$s=0$ and $s=1$.
\end{lemma}
\begin{proof}
The affine simple reflections act by
\begin{equation}\label{eq:affine-reflections}
 r_1:(\ell,d,s)\mapsto(\ell,d,-s),\qquad
 r_0:(\ell,d,s)\mapsto(\ell,d+\ell+2s,-\ell-s).
\end{equation}
At level two, their products translate the
charge by $2$ or $-2$ without changing level, so every even charge can be moved to $0$ and
every odd charge to $1$. Since the norm is preserved, \eqref{eq:N} gives $4-2d+s^2=4-2d'+(s')^2$, which yields the stated formula for $d'$.
\end{proof}

\section{The low-level cases}\label{sec:low}
The principal submatrix of $A$ on nodes $0,1$ is the Cartan matrix of $A_1^{(1)}$, giving an affine subalgebra of $g(A)$.  The action of $A_1^{(1)}$ (also known as $\widehat{\mathfrak{sl}}_2$) on $g(A)$ is level-preserving.  In particular, we may decompose $g(A)$ as
$$g(A) = \bigoplus\limits_{\ell \in \Z} g_{(\ell)}$$
where each $g_{(\ell)}$ is an $A_1^{(1)}$-module.  With our coordinates, the affine central element $h_0+h_1$ acts on $g_{(\ell)}$ by $-\ell$, and $h_1$ acts on charge $s$ by $2s$.  In this section, we obtain characters for levels $1$ and $2$ explicitly. For fixed-level characters, we omit the common factor $t^\ell$.

Real roots have norm $2$ and multiplicity $1$. By Lemma \ref{lem:full-reduction}, every positive null root is Weyl conjugate to $d(\alpha_0+\alpha_1)$ with $d>0$, whose root space belongs to the affine subalgebra $A_1^{(1)}$ and has dimension $1$ \cite[Section 4, p.~115]{FF}. Thus real and null roots attain equality, since $p(0)=p(1)=1$.

\subsection{Level one: the basic character}\label{sec:level1}
The level one module $g_{(1)}$ is the irreducible integrable basic lowest-weight module, generated by $e_{-1}$ \cite{FF}.  We will write $V$ for this module for brevity. Its lowest-weight vector has depth $0$ and eigenvalues $(-1,0)$ under $(h_0,h_1)$.  Write
\begin{equation}\label{eq:partition-product}
 P(q)=\prod_{n\ge1}(1-q^n)^{-1}=\sum_{n\ge0}p(n)q^n.
\end{equation}
The basic character formula gives the multiplicities $p(d-s^2)$, after applying the Chevalley involution to use our positive-root convention.  The character is then
\begin{equation}\label{eq:basic}
v(q,z)=\sum\limits_{d \geq 0} \sum\limits_{s \in \Z} p(d-s^2) q^d z^s =  P(q)\sum_{s\in\Z}q^{s^2}z^s.
\end{equation}
Since $N(1,d,s)=d-s^2$, every level one root attains equality in the bound.

\subsection{Irreducible affine modules and their characters}
The following is a review of standard concepts.  See for example \cite[Chapter 10]{Kac}.

Given integers $k \geq 1$ and $0 \leq j \leq k$, the usual irreducible integrable highest-weight module of $A_1^{(1)}$ has eigenvalues $(k-j,j)$ under $(h_0,h_1)$ on its highest-weight vector.  To match our positive-root convention, write $L_{k,j}$ for its Chevalley transform, the lowest-weight module with eigenvalues $(-(k-j),-j)$ on its lowest-weight vector. The central element $h_0+h_1$ acts by $-k$.

Put the lowest-weight vector at relative depth $0$, and let $L_{k,j}[n,s]$ be the subspace at relative depth $n$ with $h_1$-eigenvalue $2s$. Here $e_0$ increases relative depth by $1$ and $e_1$ preserves it. Define

$$\chi_{k,j} (q,z) = \sum\limits_{n \geq 0} \sum\limits_{s \in j/2 + \Z} \dim L_{k,j} [n,s] q^n z^s.$$

At relative depth $0$, the generators $e_1,f_1,h_1$ give the $(j+1)$-dimensional $\mathfrak{sl}_2$ representation, so
$$[q^0]\chi_{k,j}(q,z)=\sum_{a=0}^{j}z^{-j/2+a}.$$
Thus $\chi_{1,0}=v$ and $[q^0]\chi_{2,2}=z^{-1}+1+z$. A copy of $L_{k,j}$ starting at actual depth $d_0$ has character $q^{d_0}\chi_{k,j}(q,z)$, with the level factor omitted.

The coefficients of $\chi_{k,j}$ can be obtained from the Weyl-Kac character formula.  More details on characters can be found in Appendix \ref{app:characters}.

\subsection{The exact level-two character}
The following level-two construction is due to Feingold and Frenkel \cite[Section 4]{FF}. There is a surjection
$$
 \Lambda^2V\longrightarrow g_{(2)},\qquad u\wedge v\longmapsto[u,v].
$$
If we know the kernel of this surjection, then we can write exact character formulas for $g_{(2)}$.  Indeed, let $R \subset \Lambda^2V$ be the affine submodule generated by
$$w = e_{-1} \wedge [e_0, e_{-1}].$$
This vector represents the Serre relation $[e_{-1},[e_{-1},e_0]]=0$. The Serre relations of levels $0$ and $1$ already define the affine subalgebra and $V$. The remaining ideal in the free Lie algebra on $V$ is generated by $R$. Since bracketing with positive-level elements increases level, its level-two part is exactly $R$. Therefore,
$$g_{(2)} \cong\Lambda^2V/R.$$
The vector $w$ is nonzero in $\Lambda^2V$, has coordinates $(2,1,-1)$, and is killed by $f_0$ and $f_1$ under the diagonal affine action. Its eigenvalues under $(h_0,h_1)$ are $(0,-2)$. Since $V$ is integrable, $e_0,e_1,f_0,f_1$ act locally nilpotently on $\Lambda^2V$. Thus the submodule generated by $w$ is the irreducible lowest-weight module $L_{2,2}$, as identified in \cite[Section 4, pp.~115--116]{FF}. The general integrable-module theory used here is from \cite[Chapter 10]{Kac}.

One important detail is that $w$ has depth $1$, so the relative depth of $L_{2,2}$ is shifted by one. Its actual character is
\begin{equation}\label{eq:relation-character}
                    r(q,z):=\ch R=q\,\chi_{2,2}(q,z).
\end{equation}
By writing an explicit basis, we have
$$
             \ch\Lambda^2V=\frac{v(q,z)^2-v(q^2,z^2)}2.
$$
Therefore, writing $f_2(q,z)=\ch g_{(2)}$, we obtain
\begin{equation}\label{eq:f2}
 f_2(q,z)=\frac{v(q,z)^2-v(q^2,z^2)}2-r(q,z).
\end{equation}
This provides an exact character formula.  Unfortunately, due to alternating signs it is not as trivial to study as the level one formula.  We can however refine this.

\subsection{Level two: the branching identity and its normalization}
A lot of the work in this section originates from \cite[Section 2]{BB} and \cite[Sections 1 and 5.3]{CKMN}.  Let $h_n$ be the number of partitions of $2n+1$ into distinct positive odd parts.  Set
\begin{equation}\label{eq:H}
 H(q)=\sum_{n\ge0}h_nq^n,\qquad A_2(q)=\frac{H(q)-1}{q},
 \qquad Q(q)=\prod_{a\ge1}(1+q^a).
\end{equation}

Since $h_0=1$, the coefficient of $q^n$ in $A_2(q)$ counts partitions of $2n+3$ into distinct positive odd parts. The branching identity states that
$$\ch\Lambda^2V = H(q) r(q,z).$$
This expresses the exterior square of $V$ as a sum of depth-shifted copies of a single affine module, with multiplicities recorded by $H(q)$. The relation submodule $R$ is the copy corresponding to the constant term of $H(q)$. We can write explicitly

$$
	H(q) = 
 \frac{q^{-1/2}}2\left\{
 \prod_{a\ge1}(1+q^{a-1/2})-
 \prod_{a\ge1}(1-q^{a-1/2})\right\}.
$$
One can check that a monomial in this expression selects an odd number of distinct odd integers $2a - 1$ with sum $2n + 1$.

We now specialize to charge $0$ and charge $1$ as per Lemma \ref{lem:level2-charge-reduction}.  This is equivalent to selecting the coefficients of $z^0$ and $z^1$.  Then we have (\cite[eq. (5.24)]{CKMN})
\begin{equation}\label{eq:level2strings}
 [z^1]\chi_{2,2}(q,z)=\frac{E(q)}{Q(q)},\qquad
 [z^0]\chi_{2,2}(q,z)=\frac{O(q)}{Q(q)},
\end{equation}
where $E(q)=\sum_{n\ge0}p(2n)q^n$ and
$O(q)=\sum_{n\ge0}p(2n+1)q^n$.  A derivation in our notation is included in the Appendix.

Since $f_2=(H-1)r$, $H-1=qA_2$ and $r=q\chi_{2,2}$, we have

\begin{equation}\label{eq:level2-root-strings}
 [z^1]f_2(q,z)=q^2 \frac{A_2(q)E(q)}{Q(q)},\qquad
 [z^0]f_2(q,z)=q^2 \frac{A_2(q)O(q)}{Q(q)},
\end{equation}

For $s=0$, $N(2,d,0)=2d-3$. For $s=1$, $N(2,d,1)=2d-4$. The two series therefore cover odd and even values of $N$, respectively. Define
$$m_2(2b)=m(2,b+2,1),\qquad m_2(2b+1)=m(2,b+2,0),\qquad b\geq0,$$
and put

$$M_2(x) = \sum\limits_{N \geq 0} m_2(N)x^N.$$
Removing the common depth shift $q^2$ and substituting $q=x^2$ gives

\begin{equation}\label{eq:level2-parity-strings}
 \sum\limits_{b\geq 0}m_2(2b)x^{2b}= \frac{A_2(x^2)E(x^2)}{Q(x^2)},\qquad
 \sum\limits_{b\geq 0}m_2(2b+1)x^{2b+1}= x\frac{A_2(x^2)O(x^2)}{Q(x^2)}.
\end{equation}
Therefore,
$$M_2(x) = \frac{A_2(x^2)}{Q(x^2)} \left( E(x^2) + x O(x^2) \right) = P(x) \frac{A_2 (x^2)}{Q(x^2)}.$$

Here we have a relation that compares multiplicities of $g_{(2)}$ directly to the partition function. Cancelling factors in the products gives
$$\frac{P(x)}{Q(x^2)} = Q(x) P(x^4)$$
hence $P(x) = Q(x)P(x^4)Q(x^2)$.  Therefore,
$$P(x) - M_2(x) = Q(x)P(x^4)(Q(x^2) - A_2(x^2)).$$

To compare the partition function with the multiplicities, it is therefore enough to compare partitions of $2n+3$ into distinct positive odd parts with partitions of $n$ into distinct positive parts. We have the following lemma:

\begin{lemma}\label{lem:partition-injection}
For every integer $n\geq0$, there is an injection from partitions of $2n+3$ into distinct
positive odd parts to partitions of $n$ into distinct positive parts.
The partitions missing from its image are exactly
those of length $2r$ or $2r+1$, for $r\ge2$, with
$\mu_{r-1}-\mu_r=1$, where the parts $\mu_i$ are in decreasing order.
\end{lemma}

\begin{proof}
Begin with a partition $(a_1,\ldots,a_k)$ of $2n+3$ into distinct positive odd parts, in decreasing order. Its length $k=2r+1$ must be odd. If $r=0$, send it to $(n)$, interpreted as the empty partition when $n=0$.

If $r\geq1$, subtract $1$ from each $a_i$ and divide by $2$. This gives a decreasing sequence $(b_1,\ldots,b_{2r+1})$ of distinct nonnegative integers with sum $n+1-r$. Add $1$ to the first $r-1$ entries, and remove a terminal zero if present. This gives a partition of $n$ into $2r$ or $2r+1$ distinct positive parts.

To invert the construction, start from a partition $(\mu_1,\ldots,\mu_k)$ of $n$ into distinct positive parts. Lengths $0$ and $1$ are covered by the $r=0$ case. Otherwise, the length determines $r\geq1$ by $k=2r$ or $k=2r+1$. Append $0$ when the length is even, then subtract $1$ from the first $r-1$ entries. For $r=1$ there is nothing to subtract. For $r\geq2$, the entries remain distinct exactly when $\mu_{r-1}-\mu_r\geq2$. Doubling all entries and adding $1$ then uniquely recovers the partition of $2n+3$ into distinct positive odd parts. This proves injectivity and the stated description of the missing partitions.
\end{proof}

The lemma proves $Q(q)\succeq A_2(q)$ and hence the level-two bound. To describe the difference explicitly, put $(q;q)_L=\prod_{j=1}^{L}(1-q^j)$. Partitions into $L$ distinct positive parts have generating function $q^{L(L+1)/2}/(q;q)_L$. Requiring $\mu_{r-1}-\mu_r=1$ fixes the extra gap at that position to zero, removing the factor $(1-q^{r-1})^{-1}$. Therefore,

\begin{equation}\label{eq:partition-defect}
 Q(q)-A_2(q)=\sum_{r\ge2}(1-q^{r-1})
 \left\{\frac{q^{r(2r+1)}}{(q;q)_{2r}}+
        \frac{q^{(r+1)(2r+1)}}{(q;q)_{2r+1}}\right\}.
\end{equation}

The first missing partition is $(4,3,2,1)$, of weight $10$, so $Q(q)-A_2(q)$ begins with $q^{10}$. Thus $P(x)-M_2(x)$ vanishes through degree $19$. Since $Q(x)P(x^4)$ has a positive coefficient in every nonnegative degree, the coefficient of $x^{20}$ in $Q(x^2)-A_2(x^2)$ makes the difference strictly positive in every degree $N\geq20$. This proves that level-two equality occurs exactly for $0\leq N\leq19$.

\section{The Free Lie bound}\label{sec:analytical}\label{sec:tails}

So far we have been able to write explicit character formulas for levels $1$ and $2$.  At higher levels, formulas get much more complicated, and we cannot do so as easily.  To establish a comparison with the partition function, we will instead use estimates.  Our comparison uses an explicit theorem of Banerjee--Paule--Radu--Zeng \cite[Theorem 1.3]{BPRZ}:
\begin{equation}\label{eq:partition-bound}
 p(n)>\frac{e^{\Aconst\sqrt n}}{4\sqrt3\,n}
        \left(1-\frac1{2\sqrt n}\right),\qquad n\ge1,
\end{equation}
where $\Aconst = \pi \sqrt{2/3}$. We will bound $m(\ell,d,s)$ above by the right hand side with $n=N(\ell,d,s)$ when $\ell\geq3$, $d\geq165\ell$ and $|s|\leq\ell/2$, proving the strict bound in this region.

\subsection{A bound from the free Lie algebra}
Since $\nn=\bigoplus_{\ell>0} g_{(\ell)}$ is generated by its level one module $V$, there is a graded surjection from the free Lie algebra $L(V)$ on $V$ to $\nn$.  This gives an upper bound on root multiplicities. For a positive integer $\ell$ and an integer $s$, define
\begin{equation}\label{eq:theta-ell}
 \Theta_{\ell,s}(q)=
       \sum_{\substack{a_1,\ldots,a_\ell\in\Z\\a_1+\cdots+a_\ell=s}}q^{a_1^2+\cdots+a_\ell^2}.
\end{equation}

\begin{lemma}\label{lem:free-lie}
For every positive integer $\ell$, nonnegative integer $d$ and integer $s$,
\begin{equation}\label{eq:free-upper}
 m(\ell,d,s)\le\frac1\ell[q^d]P(q)^\ell\Theta_{\ell,s}(q).
\end{equation}
\end{lemma}

\begin{proof}
Let $L_\ell(V)$ be the part of $L(V)$ of bracket length $\ell$, and write $b_\ell(q,z)=\ch L_\ell(V)$, with the level factor omitted. By the surjection, we have
$$m(\ell,d,s) \leq [q^d z^s] b_\ell(q,z).$$
In the tensor algebra, we already have
$$b_\ell(q,z) \preceq v(q,z)^\ell$$
but this is not strong enough.  We require the factor $\frac{1}{\ell}$. To achieve this, use $U(L(V))\cong T(V)$ to write
$$\ch U(L(V)) = \sum\limits_{j \geq 0} t^j v(q,z)^j = \frac{1}{1 - t v(q,z)} = (1-tv)^{-1}.$$
We now use the logarithmic extraction from a PBW product identity employed
by Bauer and Bernard \cite[Section 1.2, equations (7)--(17)]{BB}.
$$\log \ch U(L(V)) = - \log(1 - tv) = \sum\limits_{\ell \geq 1} \frac{t^{\ell} v^{\ell}}{\ell}$$
hence $$[t^{\ell}] \log \ch U(L(V)) = \frac{v(q,z)^\ell}{\ell}.$$
On the other hand, if $b_{j,d,s} = \dim L_j(V)_{d,s}$, we can use the PBW theorem with an ordered homogeneous basis to write an explicit character for $U(L(V))$.  Its finite ordered monomials
$$x_1^{n_1} x_2^{n_2} ... x_r^{n_r}$$
form a basis of $U(L(V))$. The powers of a basis element $x$ of degree $(j,d,s)$ contribute the factor
$$1 + t^j q^d z^s + t^{2j} q^{2d} z^{2s} + ... = \frac{1}{1 - t^j q^d z^s}.$$
Since we have $b_{j,d,s}$ of these terms, we then have
$$\ch U(L(V)) = \prod\limits_{j \geq 1} \prod\limits_{d,s} (1 - t^j q^d z^s)^{- b_{j,d,s}}.$$
Now again take logarithms :
\begin{align*}
\log \ch U(L(V)) &=  \sum\limits_{j,d,s} b_{j,d,s} (- \log (1 - t^j q^d z^s))\\
&= \sum\limits_{j,d,s} b_{j,d,s} \sum\limits_{k \geq 1} \frac{t^{jk} q^{dk} z^{sk}}{k} \\
&= \sum\limits_{j \geq 1} \sum\limits_{k \geq 1} \frac{t^{jk}}{k} b_j(q^k, z^k).
\end{align*}
The coefficient at level $\ell$ is
\begin{align*}
[t^\ell] \log \ch U(L(V)) &= \sum\limits_{k | \ell} \frac{1}{k} b_{\ell / k} (q^k, z^k) \\
&= b_\ell (q,z) + \sum_{\substack{k | \ell \\ k \geq 2}} \frac{1}{k} b_{\ell / k} (q^k, z^k).
\end{align*}
Since all coefficients are nonnegative, comparing the two expressions gives
$$b_\ell(q,z) \preceq \frac{1}{\ell} v(q,z)^\ell.$$
Finally, \eqref{eq:basic} gives $[z^s]v(q,z)^\ell=P(q)^\ell\Theta_{\ell,s}(q)$, proving the lemma.
\end{proof}

We now estimate the coefficient in the lemma analytically. Fix $0<r<1$. Both $P(q)$ and $\Theta_{\ell,s}(q)$ converge absolutely for $|q|<1$. Using Cauchy's formula on $q=re^{i\theta}$, the triangle inequality and $|\Theta_{\ell,s}(re^{i\theta})|\leq\Theta_{\ell,s}(r)$, we have
\begin{align*}
\frac1\ell[q^d]P(q)^\ell\Theta_{\ell,s}(q) &= \frac{1}{2 \ell \pi i} \oint_{|q|=r} \frac{P(q)^\ell \Theta_{\ell,s}(q)}{q^{d+1}} dq \\
&= \frac{r^{-d}}{2 \ell \pi} \int_{- \pi}^{\pi} P(r e^{i \theta})^\ell \Theta_{\ell,s}(r e^{i \theta}) e^{- i d \theta} d \theta \\
&\leq r^{-d} \Theta_{\ell,s}(r) \frac{1}{2 \ell \pi} \int_{-\pi}^{\pi} |P(r e^{i \theta}) |^\ell d \theta \\
&= r^{-d} P(r)^\ell \Theta_{\ell,s}(r) \frac{1}{2 \ell \pi} \int_{-\pi}^{\pi} \left| \frac{P(r e^{i \theta})}{P(r)} \right|^\ell d \theta.
\end{align*}
For $\beta>0$, set
$$
 I_\ell(\beta)=\frac1{2\pi}\int_{-\pi}^{\pi}
       \left|\frac{P(re^{i\theta})}{P(r)}\right|^\ell d\theta, \qquad r = e^{-\beta}
$$
then
\begin{equation}\label{eq:analytic-cauchy}
 m(\ell,d,s)\le\frac{e^{\beta d}}\ell
 P(e^{-\beta})^\ell\Theta_{\ell,s}(e^{-\beta})I_\ell(\beta).
\end{equation}

The following three subsections give bounds on these terms, valid for $\ell \geq 3$ and $0 < \beta \leq 0.1$.  Here $\beta$ is a free parameter that determines the radius of the integration circle. The following function will appear:
\begin{equation}\label{eq:theta-remainder}
 R(\beta)=1+\frac{2u_\beta}{1-u_\beta},
 \qquad u_\beta=e^{-\pi^2/(2\beta)}.
\end{equation}

\subsection{The Gaussian bound}\label{sec:gaussian-bound}
\begin{lemma}[Gaussian bound]\label{lem:gaussian-bound}
For every $\beta>0$, positive integer $\ell$ and integer $s$, with
$R(\beta)$ as in \eqref{eq:theta-remainder},
\begin{equation} \label{eq:theta-uniform}
 \Theta_{\ell,s}(e^{-\beta})\le
 \frac{e^{-\beta s^2/\ell}}{\sqrt\ell}
 \left(\frac\pi\beta\right)^{(\ell-1)/2}R(\beta)^{\ell-1}.
\end{equation}
\end{lemma}

\begin{proof}
The one-dimensional Gaussian Poisson identity \cite{DLMFtheta} is
\[
 \sum_{n\in\Z}e^{-a(n-h)^2}
 =\sqrt{\frac\pi a}\sum_{k\in\Z}e^{-\pi^2k^2/a}e^{2\pi ikh}
 \qquad(a>0,\ h\in\mathbb R).
\]
Taking absolute values and using $0<a\leq2\beta$ and $k^2\geq k$ for $k\geq1$, we have $e^{-\pi^2k^2/a}\leq u_\beta^k$. Summing the geometric series gives
\begin{equation}\label{eq:shifted-gaussian}
 \sum_{n\in\Z}e^{-a(n-h)^2}
 \le\sqrt{\frac\pi a}
       \left(1+2\sum_{k\ge1}e^{-\pi^2k^2/a}\right)
 \le\sqrt{\frac\pi a}\,R(\beta).
\end{equation}
The proof now proceeds by induction on $\ell$, for all integer charges simultaneously. For $\ell=1$, the claimed bound is equality:
$$\Theta_{1,s}(e^{-\beta}) = e^{-\beta s^2}.$$
For $\ell\geq2$, applying the induction hypothesis at charge $s-n$ gives
\begin{align*}
 \Theta_{\ell,s}(e^{-\beta})
 &=\sum_{n\in\Z}e^{-\beta n^2}\Theta_{\ell-1,s-n}(e^{-\beta}) \\ &\leq \sum\limits_{n \in \Z} e^{-\beta n^2} \frac{e^{-\beta (s-n)^2/(\ell-1)}}{\sqrt{\ell-1}}
 \left(\frac\pi\beta\right)^{(\ell-2)/2}R(\beta)^{\ell-2} \\
 &=  \frac{R(\beta)^{\ell-2}}{\sqrt{\ell-1}}
 \left(\frac\pi\beta\right)^{(\ell-2)/2} \sum\limits_{n \in \Z} e^{-\beta\left(n^2+\frac{(s-n)^2}{\ell-1}\right)}.
\end{align*}
We can rewrite by a simple rearrangement
\[
 n^2+\frac{(s-n)^2}{\ell-1}
 =\frac{\ell}{\ell-1}\left(n-\frac s\ell\right)^2
       +\frac{s^2}{\ell}
\]
then use \eqref{eq:shifted-gaussian} with $a=\beta\ell/(\ell-1)\leq2\beta$ and $h=s/\ell$ to obtain
\[
 \frac{R(\beta)^{\ell-2}}{\sqrt{\ell-1}}
 \left(\frac\pi\beta\right)^{(\ell-2)/2}
 e^{-\beta s^2/\ell}
 \sqrt{\frac{\pi(\ell-1)}{\beta\ell}}\,R(\beta) =  \frac{e^{-\beta s^2/\ell}}{\sqrt\ell}
 \left(\frac\pi\beta\right)^{(\ell-1)/2}R(\beta)^{\ell-1},
\]
as desired.
\end{proof}

\subsection{The circle bound}\label{sec:circle-bound}
\begin{lemma}[Circle bound]\label{lem:circle-bound}
For every integer $\ell\geq3$ and $0<\beta\leq1/10$, with $r=e^{-\beta}$,
\begin{equation}\label{eq:circle-final}
 I_\ell(\beta) = \frac1{2\pi}\int_{-\pi}^{\pi}
       \left|\frac{P(re^{i\theta})}{P(r)}\right|^\ell d\theta \le\frac{1.01\,\beta^{3/2}}{\sqrt{3\pi\ell}}.
\end{equation}
\end{lemma}

\begin{proof}
Write
$$\log P(q) = \sum\limits_{n,k \geq 1} \frac{q^{nk}}{k}$$
then at $q = r e^{i \theta}$ we have, from the real part of $\log P(q)$
$$\log | P(r e^{i \theta}) | = \sum\limits_{n,k \geq 1} \frac{e^{-n k \beta}}{k} \cos(n k \theta)$$
therefore
$$\log P(r) - \log | P(r e^{i \theta}) | = \sum\limits_{n,k \geq 1} \frac{e^{-n k \beta}}{k} (1 - \cos(n k \theta)).$$
Each summand is nonnegative, so retaining only $k=1,2$ gives a lower bound on the logarithmic difference. Define, for $x>0$ and real $y$,
$$S(x,y) = \sum\limits_{n \geq 1} e^{-nx}(1 - \cos (ny)),$$
then
$$\log P(r) - \log | P(r e^{i \theta}) | \geq S(\beta, \theta) + \frac{1}{2} S(2 \beta, 2 \theta) := J_{\beta}(\theta).$$
From there, we have
$$ \left|\frac{P(re^{i\theta})}{P(r)}\right|^\ell \leq e^{- \ell J_\beta (\theta)}$$
so if we know something about $J_\beta$ then we can bound the integral.

We first establish the following two-region damping estimate for
$0<\beta\le1/10$:
\begin{equation}\label{eq:two-regions}
 J_\beta(\theta)\ge
 \begin{cases}
 \displaystyle\frac{3\theta^2}{4\beta^3},&|\theta|\le2\beta/3,\\[3pt]
 \displaystyle\frac1{3\beta},&2\beta/3\le|\theta|\le\pi.
 \end{cases}
\end{equation}
An explicit formula for $S(x,y)$ is
\begin{equation}\label{eq:S}
 S(x,y)=\frac{\coth(x/2)e^{-x}(1-\cos y)}
 {(1-e^{-x})^2+2e^{-x}(1-\cos y)}.
\end{equation}
This formula follows by summing the geometric series for $\sum_{n\geq1}e^{-nx}$ and $\sum_{n\geq1}e^{-nx+iny}$, then taking the real part of the latter.

Then it is enough to bound the numerator and the denominators in each region.  

If $|\theta| \leq 2 \beta / 3$ then we have $| \theta |, | 2 \theta | \leq 2/15$.
Taylor's inequality gives
$1-\cos y\geq y^2/2-y^4/24\geq0.499y^2$ and $2(1-\cos y)\leq y^2$. We also use $e^{-\beta}\geq0.9$, $e^{-2\beta}\geq0.8$, $1-e^{-x}\leq x$ and $\coth(x/2)\geq2/x$. Substituting these bounds into \eqref{eq:S} gives
\begin{align*}
J_\beta(\theta)
&\geq\frac{0.998(0.9+0.8/4)\theta^2}{\beta(\beta^2+\theta^2)}\\
&\geq\frac{0.998(0.9+0.8/4)}{1+4/9}\frac{\theta^2}{\beta^3}
\geq\frac{3\theta^2}{4\beta^3}.
\end{align*}

Next assume $2\beta/3\leq|\theta|\leq\pi$. At the shared endpoint, the first-region estimate gives
$$J_\beta(2\beta/3)\geq\frac{3}{4\beta^3}\left(\frac{2\beta}{3}\right)^2=\frac{1}{3\beta}.$$
The expression \eqref{eq:S} is increasing in $1-\cos y$, so $S(x,y)$ is nondecreasing for $0\leq y\leq\pi$. Thus both $S(\beta,\theta)$ and $S(2\beta,2\theta)$ are nondecreasing in $|\theta|$ up to $\pi/2$, proving the bound on $2\beta/3\leq|\theta|\leq\pi/2$. For $\pi/2\leq|\theta|\leq\pi$, discard the second summand and use
$$J_\beta (\theta) \geq S(\beta, \theta) \geq S(\beta, \pi / 2) = \frac{\coth(\beta / 2) e^{-\beta}}{1 + e^{-2 \beta}} \geq \frac{e^{- \beta}}{\beta} \geq \frac{0.9}{\beta} > \frac{1}{3 \beta}.$$
This establishes \eqref{eq:two-regions}.

We may now integrate the bound to obtain
\begin{align*}
I_\ell (\beta) &\leq \frac{1}{2 \pi} \int_{|\theta| \leq 2 \beta / 3} e^{- 3\ell \theta^2 / (4 \beta^3)} d \theta + \frac{1}{2 \pi} \int_{2 \beta / 3 \leq | \theta | \leq \pi} e^{- \ell / (3 \beta)} d \theta \\
&\leq \frac{1}{2 \pi} \int_{-\infty}^{\infty} e^{- 3\ell \theta^2 / (4 \beta^3)} d \theta + \frac{1}{2 \pi} \int_{- \pi}^\pi e^{- \ell / (3 \beta)} d \theta
\\ &= \frac{\beta^{3/2}}{\sqrt{3 \pi \ell}} + e^{-\ell / (3 \beta)}.
\end{align*}
All that is left is to absorb the second term into the first with a factor of $0.01$. Their ratio is
$$T(\ell,\beta)=\frac{\sqrt{3\pi\ell}}{\beta^{3/2}}e^{-\ell/(3\beta)}.$$
For $\ell\geq3$ and $0<\beta\leq0.1$, its logarithmic derivatives satisfy
$$\partial_\ell\log T=\frac{1}{2\ell}-\frac{1}{3\beta}<0,\qquad
\partial_\beta\log T=-\frac{3}{2\beta}+\frac{\ell}{3\beta^2}>0.$$
Thus the ratio is largest at $\ell=3$, $\beta=0.1$, where
$$T(3,0.1)=30\sqrt{10\pi}\,e^{-10}<0.007634<0.01.$$
Hence
$$e^{-\ell / (3 \beta)} < 0.01 \frac{\beta^{3/2}}{\sqrt{3 \pi \ell}}.$$
This proves \eqref{eq:circle-final}.
\end{proof}

\subsection{The partition-product bound}\label{sec:product-bound}
\begin{lemma}[Partition-product bound]\label{lem:product-bound}
For $0<\beta\leq1/10$, with $R(\beta)$ as in \eqref{eq:theta-remainder},
\begin{equation} \label{eq:product-bound}
 R(\beta)P(e^{-\beta})<1.001\sqrt{\frac\beta{2\pi}}
                  e^{\pi^2/(6\beta)-\beta/24}.
\end{equation}
\end{lemma}

\begin{proof}
Begin with the eta transformation \cite{DLMFeta} to obtain
\begin{equation}\label{eq:eta-transform}
 P(e^{-\beta})=\sqrt{\frac\beta{2\pi}}\,
 e^{\pi^2/(6\beta)-\beta/24}P(e^{-4\pi^2/\beta}).
\end{equation}
We also make use of, for $0 < x < 1$,
$$\log P(x) = \sum\limits_{n,k \geq 1} \frac{x^{nk}}{k} \leq \sum\limits_{n \geq 1} \frac{x^n}{1 - x^n} \leq \frac{1}{1-x} \sum\limits_{n \geq 1} x^n = \frac{x}{(1-x)^2}$$
hence
$$1 \leq P(x) \leq \exp \left( \frac{x}{(1-x)^2} \right).$$
Both $R(\beta)$ and $P(e^{-4\pi^2/\beta})$ increase with $\beta$, so their product is largest at $\beta=1/10$. Applying the preceding bound on $\log P$ at this endpoint gives
\begin{equation}\label{eq:radial-remainder}
 R(\beta)P(e^{-4\pi^2/\beta})
 \le\left(1+\frac{2u}{1-u}\right)e^{v/(1-v)^2}<1.001,
 \quad u=e^{-5\pi^2},\quad v=e^{-40\pi^2}.
\end{equation}
Combining this with \eqref{eq:eta-transform} proves \eqref{eq:product-bound}.
\end{proof}

\subsection{Putting it all together}\label{sec:assembly}
Recall $\Aconst = \pi \sqrt{2/3}$. Assume integer coordinates $(\ell,d,s)$ with $\ell\geq3$, $d\geq165\ell$ and $|s|\leq\ell/2$. Then $N = \ell d - \ell^2 + 1 - s^2$ and we have
\begin{equation}\label{eq:tail-domain}
 N\ge\frac{655}{4}\ell^2+1,\qquad
 N+\ell^2-1=\ell d-s^2\ge\frac{659}{4}\ell^2.
\end{equation}

\begin{proposition}[Common analytical estimate]\label{prop:common-tail}
Under these assumptions,
\begin{equation}\label{new:eq:general-tail}
 m(\ell,d,s)<\frac32(1.001)^\ell 2^{-(\ell-1)/2}
             \frac{e^{\Aconst\sqrt{N+\ell^2-1}}}{4\sqrt3\,N}.
\end{equation}
\end{proposition}

\begin{proof}
We finally choose our value of $\beta$ to be
\[
 \beta=\frac{\pi\ell}{\sqrt{6(N+\ell^2-1)}}
       \le\pi\sqrt{\frac{2}{1977}}<0.1.
\]
This is where we use $d \geq 165 \ell$.  Then all three bounds apply.  Substitute \eqref{eq:theta-uniform},
\eqref{eq:circle-final}, and \eqref{eq:product-bound} into
\eqref{eq:analytic-cauchy}. The $\Theta$ bound contributes $R(\beta)^{\ell-1}$. Since $R(\beta)\geq1$, we may bound $R(\beta)^{\ell-1}P(e^{-\beta})^\ell$ above by $[R(\beta)P(e^{-\beta})]^\ell$. Drop $e^{-\ell\beta/24}\leq1$ from the partition-product bound and simplify to obtain
\begin{equation}\label{eq:common-before-saddle}
 m(\ell,d,s)\le
 \frac{1.01\,\beta^2(1.001)^\ell}{2^{\ell/2}\sqrt3\,\pi\ell^2}
 \exp\!\left(\frac{\beta(N+\ell^2-1)}{\ell}
                   +\frac{\pi^2\ell}{6\beta}\right).
\end{equation}
Substituting $\beta$ makes the exponent $\Aconst\sqrt{N+\ell^2-1}$. Using $\beta^2/\ell^2=\pi^2/[6(N+\ell^2-1)]\leq\pi^2/(6N)$ gives
\begin{align*}
m(\ell,d,s)&\leq\frac{1.01\pi(1.001)^\ell}{2^{\ell/2}6\sqrt3\,N}e^{\Aconst\sqrt{N+\ell^2-1}}\\
&=\frac{1.01\pi\sqrt2}{3}\frac{(1.001)^\ell}{2^{(\ell-1)/2}4\sqrt3\,N}e^{\Aconst\sqrt{N+\ell^2-1}}.
\end{align*}
Finally, for a simpler form, we have
\[
 \frac{1.01\pi\sqrt2}{3}<\frac32
\]
which proves the proposition.
\end{proof}

From there it remains only to compare to \eqref{eq:partition-bound}, which is the reason the constants were split instead of combined.  Begin with
\[
 \Aconst\bigl(\sqrt{N+\ell^2-1}-\sqrt N\bigr)
 =\frac{\Aconst(\ell^2-1)}{\sqrt{N+\ell^2-1}+\sqrt N}
 <\frac{\Aconst}{\sqrt{655}}\left(\ell-\frac1\ell\right)
 <0.101\left(\ell-\frac1\ell\right).
\]
Also $1/(2\sqrt N)<1/(\ell\sqrt{655})$. Dividing
\eqref{new:eq:general-tail} by the partition lower bound
\eqref{eq:partition-bound} therefore gives
\begin{equation}\label{new:eq:tail-ratio}
 \frac{m(\ell,d,s)}{p(N)}<
 \frac{(3/2)(1.001)^\ell 2^{-(\ell-1)/2}
       e^{0.101(\ell-1/\ell)}}
 {1-1/(\ell\sqrt{655})}.
\end{equation}
We then have an expression that only depends on $\ell$.

\begin{proposition}\label{prop:tail}
For every $\ell\ge3$, $|s|\le\ell/2$, and $d\ge165\ell$,
$m(\ell,d,s)<p(N)$.
\end{proposition}
\begin{proof}
At $\ell=3$, the right side of \eqref{new:eq:tail-ratio} is
\[
 \frac{(3/4)(1.001)^3e^{8(0.101)/3}}
 {1-1/(3\sqrt{655})}<0.998<1.
\]
As $\ell$ increases by one, the denominator increases, while the
numerator decreases because it is multiplied by
\[
 \frac{1.001}{\sqrt2}
 e^{0.101(1+1/(\ell(\ell+1)))}
 \le\frac{1.001}{\sqrt2}e^{(0.101)13/12}<0.790<1.
\]
Thus the comparison only improves for $\ell\ge3$.
\end{proof}

This completes the proof for the large-depth region $d\geq165\ell$. The free Lie coefficient bound \eqref{eq:free-upper} does not suffice throughout the remaining region $2\ell\leq d<165\ell$. For example, at $(\ell,d,s)=(10,20,0)$, it gives the upper bound $246128486901/2$, which exceeds $p(101)=214481126$. We therefore need a sharper bound for the remaining region.

\section{The homological bound}\label{sec:homology}
In this section, we develop a methodology to study $A_1^{++}$ from parabolic homology.  This will provide much sharper bounds.

\subsection{The positive-level radical and its resolution}
Let
\begin{equation}\label{eq:radical}
 \nn=\bigoplus_{\ell>0}g_{(\ell)},\qquad B=U(\nn),\qquad U=\ch B.
\end{equation}
The level-zero part of $B$ is a copy of $\C$. Let $B_+$ be the augmentation ideal, the kernel of the algebra homomorphism $\epsilon:B\to\C$ that sends $1$ to $1$ and every element of $\nn$ to zero. Although the positive-level spaces are infinite-dimensional, $B$ has finite-dimensional pieces in each root degree, by PBW and the finite number of decompositions of a fixed positive degree.

The following construction is the multigraded version of the standard
minimal graded free resolution; see \cite[Section 1.5, Lemmas 1.5.1--1.5.2]{Rogalski}
for graded Nakayama and the minimality criterion. The proof below uses height and does not require $B$ or its kernels to be finitely generated.

\begin{lemma}[Minimal multigraded resolution]\label{lem:minimal}
The trivial left $B$-module has a multigraded free resolution
\[
 \cdots\longrightarrow F_2\longrightarrow F_1\longrightarrow
 F_0=B\longrightarrow\C\longrightarrow0,
 \qquad F_i\cong B\otimes_{\C}E_i,
\]
such that the matrix entries of $d_i:F_i\to F_{i-1}$ lie in $B_+$ for $i\geq1$ and
$E_i\cong\Tor_i^B(\C,\C)\cong H_i(\nn,\C)$ as multigraded
vector spaces. The modules $F_i$ have finite-dimensional pieces in each root degree.
\end{lemma}

\begin{proof}
We must construct $E_i$ first.  Set $E_0 = \C, F_0 = B$ and $K_0 = \ker(B \to \C) = B_+$.

We then proceed iteratively. Suppose $K_{i-1}\subseteq F_{i-1}$ is the next kernel to handle. The submodule $B_+K_{i-1}$ consists of finite sums of products $bk$ with $b\in B_+$ and $k\in K_{i-1}$. Choose a homogeneous vector-space basis $\{\bar k_a\}$ of $K_{i-1}/B_+K_{i-1}$ and representatives $k_a\in K_{i-1}$ of the same multidegrees. Introduce formal symbols $e_a$ with $\deg e_a=\deg k_a$. Define
\[
 E_i=\bigoplus_a\C e_a,\qquad
 F_i=B\otimes_{\C}E_i,\qquad
 d_i(b\otimes e_a)=bk_a\in K_{i-1}\subset F_{i-1}.
\]
By construction $F_i$ maps onto our chosen representatives.

We need to show that this map is surjective onto $K_{i-1}$.  We proceed by induction on the integer height $h=\ell+2d+s$ of a homogeneous element $k\in K_{i-1}$ of degree $(\ell,d,s)$.

Expressing the class of $k$ in the quotient basis gives
\[
 k=\sum_a c_a k_a+\sum_j b_jk'_j,
 \qquad c_a\in\C,\quad b_j\in B_+,\quad k'_j\in K_{i-1}.
\]

Both sums are finite. Taking homogeneous components, we may assume that every term has the same degree as $k$, with each $b_j$ and $k'_j$ homogeneous. For $h=0$, the second sum is absent because $B_+K_{i-1}$ has no height-zero component. Thus $k$ is a linear combination of the chosen representatives and belongs to the image of $d_i$.

Now let $h\geq1$. Each $k'_j$ has height $h-\height(b_j)<h$, since $b_j\in B_+$. By induction, each $k'_j$ is in the image of $d_i$. The map is $B$-linear, so each $b_jk'_j$ is also in the image, as is each chosen representative. This proves surjectivity.

Set $K_i=\ker d_i$ and repeat, giving exactness at every stage. Local finiteness is also preserved: $E_i$ has the graded dimensions of a quotient of $K_{i-1}$, and each degree of $B\otimes_{\C}E_i$ is a finite sum of tensor products of finite-dimensional spaces.

Next we show that the resolution is minimal. Let $\sum_a b_a\otimes e_a\in K_i$. Its image is $\sum_a b_a k_a=0$. In the quotient $K_{i-1}/B_+K_{i-1}$, this becomes
$$\sum_a\epsilon(b_a)\bar k_a=0.$$
By linear independence, each $\epsilon(b_a)=0$, hence $K_i\subseteq B_+F_i$. Together with $K_0=B_+F_0$, this gives $d_i(F_i)=K_{i-1}\subseteq B_+F_{i-1}$ for every $i\geq1$. Since $B_+$ is supported on positive levels, it also follows inductively that $E_i$ is supported on levels at least $i$ and $K_i$ on levels at least $i+1$.

Finally, tensor over $B$ with the trivial right module $\C=B/B_+$.  More explicitly we form
$$\C \otimes_B F_i \qquad c \cdot b = c \epsilon(b).$$
We obtain $\C\otimes_B F_i\cong E_i$ from the map $c \otimes (b \otimes e) \to c \epsilon(b) e$, and all differentials become
zero because their entries lie in $B_+$. The resulting homology is
therefore $E_i$, while by definition it computes
$\Tor_i^B(\C,\C)$. The identification with Lie algebra
homology $H_i(\nn,\C)$ is \cite[Corollary 7.3.6]{Weibel}; the
Chevalley--Eilenberg free resolution realizes it explicitly
\cite[Theorem 7.7.2 and Corollary 7.7.3]{Weibel}.
\end{proof}

\subsection{Odd truncation and the logarithmic majorant}
Now we have the main tool we need to state the theorem.

\begin{theorem}[Odd homology truncation]\label{thm:odd}
For an odd positive integer $m$, define
\begin{equation}\label{eq:Dm}
 D_m=\sum_{i=0}^m(-1)^i\ch H_i(\nn,\C).
\end{equation}
Then $D_m^{-1}$ and $-\log D_m$ have nonnegative coefficients, and
\begin{equation}\label{eq:majorant}
 m(\ell,d,s)\le[t^\ell q^d z^s](-\log D_m)
 \qquad(\ell>0).
\end{equation}
\end{theorem}

\begin{proof}
Let $K_m = \ker(F_m \to F_{m-1})$ from Lemma \ref{lem:minimal}.  Write $Z_m = \ch K_m$.  Then we have an exact sequence
\[
 0\longrightarrow K_m\longrightarrow F_m\longrightarrow\cdots
 \longrightarrow F_0\longrightarrow\C\longrightarrow0
\]
with, degree by degree, the Euler-characteristic identity
\cite[pp.~31--32, equations (2.C)--(2.D)]{Rogalski}
\begin{equation}\label{eq:truncated-euler}
 \sum_{i=0}^m(-1)^i\ch F_i=1+(-1)^mZ_m.
\end{equation}
Recall $U=\ch B$. Lemma \ref{lem:minimal} gives $\ch F_i=U\,\ch H_i(\nn,\C)$. Substituting into \eqref{eq:truncated-euler} and using that $m$ is odd gives
\begin{equation}\label{eq:sign-controlled}
                       D_mU=1-Z_m.
\end{equation}
By Lemma \ref{lem:minimal}, $Z_m$ has nonnegative coefficients and is supported on levels at least $m+1$. Also, $U$ and $D_m$ have constant term $1$. The following expansions are therefore well-defined in the completed character ring of Section \ref{sec:geometry}:
\begin{align}
 D_m^{-1}&=\frac{U}{1-Z_m}=U\sum_{r\geq0}Z_m^r\succeq0,\label{eq:inverse-positive}\\
 -\log D_m&=\log U-\log(1-Z_m)
            =\log U+\sum_{r\ge1}\frac{Z_m^r}{r}
            \succeq\log U.\label{eq:log-positive}
\end{align}

For $\alpha=\ell\alpha_{-1}+d\alpha_0+(d+s)\alpha_1$, write $e^\alpha=t^\ell q^d z^s$ and $\ell(\alpha)=\ell$. We may once again use the PBW theorem \cite[Theorem 7.3.7]{Weibel}, as in Lemma \ref{lem:free-lie}, applied to an ordered homogeneous basis in each finite root degree, to obtain
\begin{equation}\label{eq:pbw}
 U=\prod_{\substack{\alpha\in\Delta^+\\\ell(\alpha)>0}}
       (1-e^\alpha)^{-\dim g_\alpha},\qquad
 \log U=\sum_{\substack{\alpha\in\Delta^+\\\ell(\alpha)>0}}
       \dim g_\alpha\sum_{r\ge1}\frac{e^{r\alpha}}r.
\end{equation}
All terms on the right are nonnegative, so \eqref{eq:log-positive} also proves that $-\log D_m$ has nonnegative coefficients. For each positive-level root $\alpha$, its $r=1$ term gives
$$\dim g_\alpha\leq[e^\alpha](\log U)\leq[e^\alpha](-\log D_m)$$
or equivalently, the statement of the theorem.
\end{proof}

\begin{remark}
The statement provides a bound for all odd values of $m$. For even $m$, \eqref{eq:truncated-euler} instead gives $D_mU=1+Z_m$, so the same argument does not give these coefficientwise bounds. For $m=1$, we have $H_1(\nn,\C)\cong\nn/[\nn,\nn]\cong V$ and $D_1=1-tv$, recovering the free Lie bound of Lemma \ref{lem:free-lie}.
\end{remark}

The rest of this section specializes the bound to $m=5$. Our first objective is to calculate $D_5$ explicitly.

\subsection{Parabolic Homology}
In this section we identify the homology groups $H_1,\ldots,H_5$ with the goal of calculating $D_5$.  We use the parabolic homology formula and the notation from \cite{KK}, where these methods were applied to $A_1^{++}$ extensively.

Let $I = \{ -1, 0, 1 \}, J = \{0,1\}$ and $W_J$ be the subgroup of $W$ generated by $r_0, r_1$.   Let $\rho$ satisfy $\langle \rho, \alpha_i^\vee \rangle = 1$ for all $i$.  Following Kang--Kim \cite[Section 3.2]{KK}, write
\[
 \Delta_J^+=\Delta^+\cap\sum_{j\in J}\Z\alpha_j,
 \qquad \Delta_+(J)=\Delta^+\setminus\Delta_J^+,
\]
for the positive roots inside and outside the Levi, respectively, and set
\begin{equation}\label{eq:minimal-reps}
 W(J)=\{w\in W:w\Delta^-\cap\Delta^+\subseteq\Delta_+(J)\}.
\end{equation}
These are the minimal-length representatives of the right cosets
$W_Jw$, so $W=W_JW(J)$. For the enumeration we use the equivalent test
\[
 w\in W(J)\quad\Longleftrightarrow\quad
 w^{-1}\alpha_j\in\Delta^+\quad\text{for every }j\in J.
\]

The affine Levi is generated by $e_0,e_1,f_0,f_1$ and the full Cartan subalgebra of $g(A)$. We include the full Cartan subalgebra so that the lowest weight also specifies the depth at which each affine module starts. Write $V_J(\lambda)$ for its irreducible integrable highest-weight module with highest weight $\lambda$. The following theorem is the Kac--Moody analogue of \emph{Kostant's formula} for finite-dimensional semisimple Lie algebras \cite{Kostant61}, due to Garland and Lepowsky \cite{GL}. We use the formulation of Kang and Kim \cite[Proposition 3.4, equation (3.9)]{KK}.

\begin{theorem}[Parabolic homology]\label{thm:parabolic-homology}
For the negative-level radical $\nn_-=\bigoplus_{\ell<0}g_{(\ell)}$ and every $i\ge0$, the
parabolic homology formula \cite[Proposition 3.4, equation (3.9)]{KK} gives
\begin{equation}\label{eq:GL}
 H_i(\nn_-,\C)\cong
 \bigoplus_{\substack{w\in W(J)\\\len(w)=i}}
                  V_J(w\rho-\rho).
\end{equation}
Correspondingly, $H_i(\nn,\C)$ is the direct sum of the irreducible
lowest-weight affine modules with lowest weights
\begin{equation}\label{eq:beta}
                              \beta(w)=\rho-w\rho
\end{equation}
for $w\in W(J)$ with $\len(w)=i$, each occurring once.
\end{theorem}

\begin{proof}
Formula \eqref{eq:GL} is \cite[Proposition 3.4, equation (3.9)]{KK} with trivial coefficients. The Chevalley involution maps $\nn_-$ to $\nn$ and negates Cartan weights, so it takes each highest-weight summand to a lowest-weight module with lowest weight $\rho-w\rho$.
\end{proof}

It then suffices to enumerate the Weyl elements to obtain characters for each homology group. In products of reflections, the rightmost reflection acts first. Starting from the identity, we multiply the reflection matrices on the right by each $r_i$, retaining each distinct matrix at its shortest length, and apply the test above. There are $37$ Weyl group elements of length $0$ to $5$, of which $7$ lie in $W(J)$. Their weights in Table \ref{tab:homology} can be calculated recursively from
\[
 \beta(1)=0,\qquad \beta(wr_i)=\beta(w)+w\alpha_i.
\]

\begin{table}[H]
\centering
\begin{tabular}{@{}clrrr@{}}
\toprule
$i$ & A representative $w$ & $\beta=(\ell,d,s)$ & $j$ & $d$\\
\midrule
0 & $1$ & $(0,0,0)$ & 0 & 0\\
1 & $r_{-1}$ & $(1,0,0)$ & 0 & 0\\
2 & $r_{-1}r_0$ & $(2,1,-1)$ & 2 & 1\\
3 & $r_{-1}r_0r_1$ & $(4,3,-2)$ & 4 & 3\\
4 & $r_{-1}r_0r_1r_0$ & $(7,6,-3)$ & 6 & 6\\
5 & $r_{-1}r_0r_1r_0r_{-1}$ & $(10,10,-5)$ & 10 & 10\\
5 & $r_{-1}r_0r_1r_0r_1$ & $(11,10,-4)$ & 8 & 10\\
\bottomrule
\end{tabular}
\caption{All parabolic homology constituents through homological degree five.
The coordinates are level, depth and charge; $j=-2s$ is the horizontal highest weight.}
\label{tab:homology}
\end{table}

For a lowest weight $\beta=(\ell,d,s)$, the affine labels are
\[
 \langle\beta,\alpha_0^\vee\rangle=-\ell-2s=-(\ell-j),
 \qquad \langle\beta,\alpha_1^\vee\rangle=2s=-j.
\]
Thus each row of positive homological degree gives a copy of $L_{\ell,j}$ starting at depth $d$, with character $t^\ell q^d\chi_{\ell,j}(q,z)$. Since $\chi_{\ell,j}$ already records charge, no shift in $z$ is needed. Together with $H_0=\C$, we obtain
\begin{equation}\label{eq:d5-characters}
\begin{aligned}
D_5(t,q,z) ={}& 1 - t \chi_{1,0} + t^2 q \chi_{2,2} - t^4 q^3 \chi_{4,4}\\
 &{}+ t^7 q^6 \chi_{7,6} - t^{10} q^{10} \chi_{10,10} - t^{11} q^{10} \chi_{11,8},
\end{aligned}
\end{equation}
where all characters are evaluated at $(q,z)$.

Numerical analysis indicates that this bound is good enough in the high level region.  We simply have to prove it.

\subsection{A bound for high level}\label{sec:cutoff}
At first glance, $D_5$ appears to be quite the complicated object, and it may even look like we have not made much progress.  However recall that Theorem \ref{thm:odd} bounds the multiplicities of $A_1^{++}$ by the coefficients of $F=-\log D_5$, not $D_5$ itself.  We can estimate these coefficients using the zeroes of $D_5$.

Fix positive $q,z$ for which the affine characters converge.  Then $D_5(t,q,z)$ is a polynomial of degree $11$ in $t$ with constant term $1$.  Factor this polynomial as
\[
D_5(t,q,z) = \prod_{j=1}^{11} (1-t/\rho_j(q,z)),
\]
where the zeroes $\rho_j(q,z)$ are counted with multiplicity. They depend on the fixed pair $(q,z)$.  Taking the formal logarithm at $t=0$ gives, for every integer $\ell\geq1$,
\[
0\leq[t^\ell]F(t,q,z)
 =\frac{1}{\ell}\sum_{j=1}^{11}\rho_j(q,z)^{-\ell}
 \leq\frac{1}{\ell}\sum_{j=1}^{11}|\rho_j(q,z)|^{-\ell}.
\]
If we know something about these zeroes uniformly then we can get a uniform bound.

\begin{lemma}[Zero bounds]\label{lem:input}
Let $\ell,d,s$ be real numbers with $\ell>0$, $2\ell\leq d\leq165\ell$, and $-\ell/2\leq s\leq0$. Put $x=d/\ell$ and $u=s/\ell$, and choose
\begin{equation}\label{eq:cutoff-rule}
 \beta=\frac{\pi}{\sqrt{6(x-u^2-7/8)}},\qquad
 q=e^{-\beta},\quad z=q^{-2u},\quad
 T=e^{1/42}q^{\,x-15/8}.
\end{equation}

The series $F=-\log D_5$ converges at the evaluation point $(T,q,z)$.
Every complex zero $\rho(q,z)$ of the polynomial
$t\mapsto D_5(t,q,z)$ satisfies $|\rho(q,z)|\ge T$.
\end{lemma}
We delay the proof to Section \ref{sec:zero-bound}.  There we also show that $T$ is chosen so that, with $N=\ell d-\ell^2+1-s^2$,
\begin{equation}\label{eq:exact-exponent}
 T^{-\ell}q^{-d}z^{-s}
 \leq e^{\Aconst\sqrt{N-1}-\ell/60}.
\end{equation}
Now suppose $\ell,d,s$ are integers in the same region, and fix the corresponding $T,q,z$ from \eqref{eq:cutoff-rule}. Applying the zero bound at this same pair $(q,z)$ gives $\sum_{j=1}^{11}|\rho_j(q,z)|^{-\ell}\leq11T^{-\ell}$. By Theorem \ref{thm:odd} and the nonnegative depth and charge coefficients of $F$, we have
\[
m(\ell,d,s)q^dz^s\leq[t^\ell]F(t,q,z)
 \leq\frac{11}{\ell}T^{-\ell}.
\]
Multiplying by $q^{-d}z^{-s}$ and applying \eqref{eq:exact-exponent} gives
\[
 m(\ell,d,s)\le\frac{11}{\ell}T^{-\ell}q^{-d}z^{-s}
 \le\frac{11}{\ell}e^{\Aconst\sqrt{N-1}-\ell/60}.
\]
Since $\sqrt{N-1}\leq\sqrt N$, we can compare this with the partition bound \eqref{eq:partition-bound}. Using $3\ell^2/4+1\le N\le164\ell^2+1$ gives
\[
 \frac{m(\ell,d,s)}{p(N)}
 <\frac{44\sqrt3\,N}{\ell(1-1/(2\sqrt N))}e^{-\ell/60}
 \le\frac{44\sqrt3(164\ell^2+1)}{\ell-1/\sqrt3}e^{-\ell/60}.
\]
The logarithmic derivative of the last expression is
\[
\frac{328\ell}{164\ell^2+1}-\frac{1}{\ell-1/\sqrt3}-\frac{1}{60}
 <\frac{1}{\ell}-\frac{1}{60}<0\qquad(\ell>60).
\]
At $\ell=1000$ the expression is less than $0.723$, so it remains less than $1$ for all $\ell\geq1000$. This proves

\begin{proposition}\label{prop:homological-cutoff}
For integer coordinates with $\ell\ge1000$, $2\ell\le d\le165\ell$, and
$-\ell/2\le s\le0$, one has $m(\ell,d,s)<p(N)$.
\end{proposition}

\subsection{Uniform bounds for zeroes}\label{sec:zero-bound} 

The goal of this section is to prove Lemma \ref{lem:input}.

Put $F = - \log D_5$, whose coefficients are nonnegative by Theorem \ref{thm:odd}.  Let $x = d/\ell$ and $u = s / \ell$.  For $(x,u) \in [2,165] \times [-1/2,0]$, choose
\begin{equation}\label{eq:rule}
 \beta=\frac{\pi}{\sqrt{6(x-u^2-7/8)}},\qquad
 q=e^{-\beta},\quad z=q^{-2u},\quad
 T=e^{1/42}q^{\,x-15/8}.
\end{equation}

Observe that $(x,u)$ are in a finite box but cover all possible values in the infinite region $2 \ell \leq d \leq 165 \ell$ and $- \ell/2 \leq s \leq 0$.  This will allow us to perform a finite computation, which will be done by computer.

Put $Y = x-1-u^2$ so that $N-1 = \ell^2 Y$ and $Y\geq3/4$.

\begin{lemma}\label{lem:55-extraction}
\begin{equation}
 \beta(2Y+1/8)\le\Aconst\sqrt Y+\frac1{140}
\end{equation}
\end{lemma}

\begin{proof}
Observe that $Y \geq 3/4$ and $\beta = \Aconst/(2 \sqrt{Y+1/8})$.  Then
\begin{align*}
\beta(2Y+1/8) - \Aconst \sqrt{Y} &= \frac{\Aconst}{128 \sqrt{Y + 1/8} (\sqrt{Y + 1/8} + \sqrt{Y})^2} \\
&\leq \frac{\Aconst}{512Y \sqrt{Y + 1/8}} \leq \frac{\pi}{96 \sqrt{21}} < \frac{1}{140}
\end{align*}
as desired.
\end{proof}
Since $1/42-1/140=1/60$, we have
\begin{equation}\label{eq:rule2}
T^{-\ell} q^{-d} z^{-s} = e^{\ell (\beta(2Y + 1/8) - 1/42)} \leq e^{\Aconst \sqrt{N-1} - \ell/60}.
\end{equation}
The constants in \eqref{eq:rule} were chosen in reverse from \eqref{eq:rule2} which we know to be enough.

Before we can prove Lemma \ref{lem:input}, we need a few more technical results.  Let $\gamma=-2\beta u$, so that $z=e^{-\gamma}$, and define
\begin{equation}\label{eq:tau}
 \tau(\beta,\gamma):=\log T
 =\beta-\frac{\pi^2/6+\gamma^2/4}{\beta}+\frac1{42}.
\end{equation}

The inverse change of coordinates is
\[
 u=-\frac{\gamma}{2\beta},\qquad
 x=\frac78+\frac{\pi^2/6+\gamma^2/4}{\beta^2}.
\]
Thus the ratio rectangle $[2,165]\times[-1/2,0]$ maps exactly onto
\begin{equation}\label{eq:real-domain}
\mathcal D=\left\{(\beta,\gamma):\ \begin{gathered}
 \beta>0,\quad 0\leq\gamma\leq\beta,\\
 \frac98\beta^2\leq\frac{\pi^2}{6}+\frac{\gamma^2}{4}
 \leq\frac{1313}{8}\beta^2
\end{gathered}\right\}.
\end{equation}
Since $1/10<\beta<11/8$ and $0\leq\gamma\leq\beta$, this domain is contained in the rectangle
\[
 \mathcal Q=[1/10,11/8]\times[0,11/8].
\]

\begin{lemma}[Convexity in logarithmic coordinates]\label{lem:55-convexity}
Let
\[
 G(t,q,z)=\sum_{n\ge0,\ d,s\in\Z}c_{n,d,s}t^nq^dz^s,
 \qquad c_{n,d,s}\ge0.
\]
Suppose $G$ converges at finitely many positive points
$p_j=(t_j,q_j,z_j)$ and $G(p_j)\le M$.
For $\lambda_j\ge0$ with $\sum_j\lambda_j=1$, it converges at
\[
  p=\left(\prod_jt_j^{\lambda_j},
                 \prod_jq_j^{\lambda_j},
                 \prod_jz_j^{\lambda_j}\right),
 \qquad G(p)\le M.
\]
\end{lemma}
\begin{proof}
Let $S$ be the sum over any finite set of monomial indices. For one
index $(n,d,s)$, convexity of the real exponential gives
\begin{align*}
 c_{n,d,s}\prod_j(t_j^nq_j^dz_j^s)^{\lambda_j}
 &=c_{n,d,s}\exp\!\left(\sum_j\lambda_j
        (n\log t_j+d\log q_j+s\log z_j)\right)\\
 &\le\sum_j\lambda_j c_{n,d,s}
        \exp(n\log t_j+d\log q_j+s\log z_j).
\end{align*}
Summing this inequality over the finite set yields
\[
 S(p)\le\sum_j\lambda_jS(p_j)
 \le\sum_j\lambda_jG(p_j)
 \le M\sum_j\lambda_j=M.
\]
Convergence follows by taking the supremum of these nonnegative finite sums.
\end{proof}

In the proof of Lemma \ref{lem:input} we will interpolate in these logarithmic coordinates.  We need to bound the interpolation error (see figure \ref{fig:interpolation-allowance}).  For $B=[b_0,b_1]\times[g_0,g_1]\subset\mathcal Q$, with
$b_0<b_1$ and $g_0<g_1$, define the interpolation allowance
\begin{equation}\label{eq:error}
 E_B=\frac{(\pi^2/3+g_1^2/2)(b_1-b_0)^2}{8b_0^3}
          +\frac{(g_1-g_0)^2}{16b_0}.
\end{equation}

\Needspace{12\baselineskip}
\begin{lemma}[From rectangle vertices to the entire rectangle]
\label{lem:55-rectangle}
Let $B$ and $E_B$ be as above, let $\tau$ be given by
\eqref{eq:tau}, and let $G$ be a nonnegative series as in
Lemma~\ref{lem:55-convexity}.
If it converges and is at most $M$ at all four points
\[
 \left(e^{\tau(b_i,g_j)+E_B},e^{-b_i},e^{-g_j}\right),
 \qquad i,j\in\{0,1\},
\]
then it converges and is at most $M$ at
\[
 \left(e^{\tau(\beta,\gamma)},e^{-\beta},e^{-\gamma}\right)
 \qquad\text{for every }(\beta,\gamma)\in B.
\]
\end{lemma}
\begin{proof}
Differentiating \eqref{eq:tau} gives
\[
 \tau_{\beta\beta}
   =-\frac{\pi^2/3+\gamma^2/2}{\beta^3},\qquad
 \tau_{\gamma\gamma}=-\frac1{2\beta}.
\]
On $B$ we have $\beta\ge b_0>0$ and $0\le\gamma\le g_1$, so
\[
 |\tau_{\beta\beta}|\le
 M_\beta:=\frac{\pi^2/3+g_1^2/2}{b_0^3},\qquad
 |\tau_{\gamma\gamma}|\le M_\gamma:=\frac1{2b_0}.
\]

If $If$ linearly interpolates $f\in C^2([a,b])$ at its endpoints,
then, for $a<t<b$, the interpolation remainder gives
\[
 f(t)-If(t)=\frac{f''(\xi)}2(t-a)(t-b)
 \quad\text{for some }\xi\in(a,b).
\]
Since $(t-a)(b-t)\le(b-a)^2/4$,
\[
 |f(t)-If(t)|\le\frac{\|f''\|_\infty(b-a)^2}{8}.
\]
The same bound holds at the endpoints, where the error is zero.

Let $I_\beta$ and $I_\gamma$ be linear interpolation in the respective
coordinates, and let $I_B=I_\beta I_\gamma$ be bilinear interpolation
from the four vertices. The identity
\[
 \tau-I_B\tau=(\tau-I_\beta\tau)
                    +I_\beta(\tau-I_\gamma\tau)
\]
and the fact that $I_\beta$ takes convex averages give
\begin{align*}
 |\tau-I_B\tau|
 &\le |\tau-I_\beta\tau|
                  +I_\beta|\tau-I_\gamma\tau|\\
 &\le\frac{M_\beta(b_1-b_0)^2}{8}
             +\frac{M_\gamma(g_1-g_0)^2}{8}=E_B.
\end{align*}
In particular,
\begin{equation}\label{eq:55-interpolation}
 \tau(\beta,\gamma)\le I_B\tau(\beta,\gamma)+E_B.
\end{equation}

Fix $(\beta,\gamma)\in B$, and write
\[
 \theta=\frac{\beta-b_0}{b_1-b_0},\qquad
 \eta=\frac{\gamma-g_0}{g_1-g_0}.
\]
Both belong to $[0,1]$. The bilinear weights are
\[
 \lambda_{00}=(1-\theta)(1-\eta),\quad
 \lambda_{10}=\theta(1-\eta),\quad
 \lambda_{01}=(1-\theta)\eta,\quad
 \lambda_{11}=\theta\eta.
\]
They are nonnegative and satisfy
\[
 \sum_{i,j}\lambda_{ij}=1,\qquad
 \sum_{i,j}\lambda_{ij}b_i=\beta,\qquad
 \sum_{i,j}\lambda_{ij}g_j=\gamma,\qquad
 \sum_{i,j}\lambda_{ij}\tau(b_i,g_j)=I_B\tau.
\]
Thus the geometric average of the four assumed convergence points is
\[
 \left(e^{I_B\tau(\beta,\gamma)+E_B},e^{-\beta},e^{-\gamma}\right).
\]
Lemma~\ref{lem:55-convexity} proves convergence there, with value at
most $M$. Finally, \eqref{eq:55-interpolation} shows that the desired
first coordinate is no larger. Since the powers of $t$ in $G$ are
nonnegative, decreasing $t$ decreases every nonnegative monomial.
Consequently
\[
 G\!\left(e^{\tau(\beta,\gamma)},e^{-\beta},e^{-\gamma}\right)
 \le G\!\left(e^{I_B\tau(\beta,\gamma)+E_B},e^{-\beta},e^{-\gamma}\right)
 \le M.
\]
This termwise comparison also proves convergence at the desired point.
\end{proof}

\begin{figure}[htbp]
\centering
\includegraphics[width=0.88\linewidth]{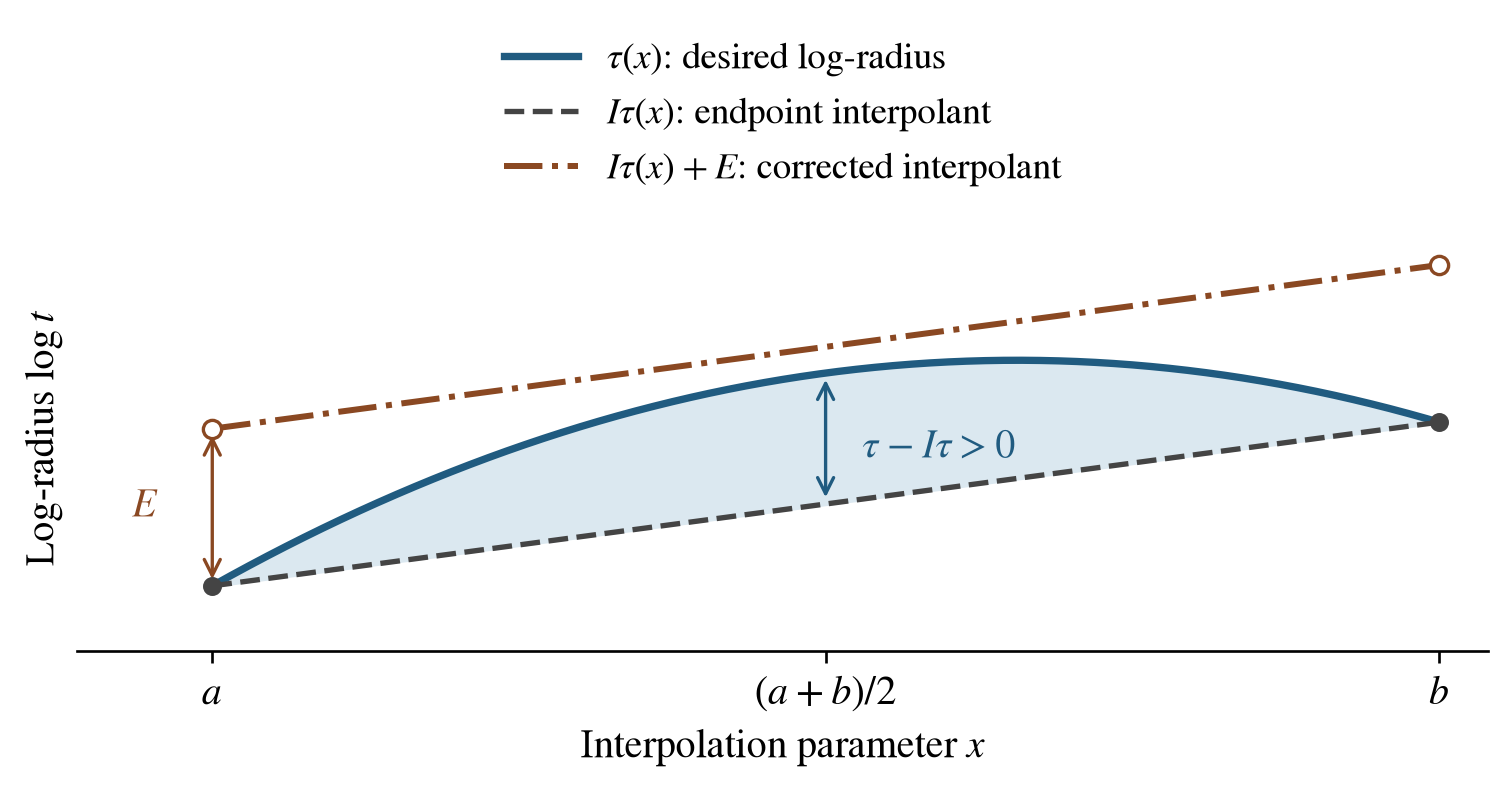}
\caption{The interpolation allowance in one dimension (schematic).
The concave curve gives the desired log-radius. The dashed chord
is controlled by uncorrected endpoint checks, but lies below the
curve. Raising both endpoint log-radii by $E$ produces an
interpolant above the curve, so monotonicity in $t$ gives the
required bound throughout the interval.}
\label{fig:interpolation-allowance}
\end{figure}

The next lemma helpful to convert a bound from a handful of points to a bound on an interval.

\begin{lemma}[Bernstein lower bounds on a whole interval]
\label{lem:55-bernstein}
Let $p$ be a real polynomial of degree at most $m$, and let $a<b$.
Write
\[
 p(a+(b-a)y)=\sum_{j=0}^{m}B_j\binom mj y^j(1-y)^{m-j}.
\]
If certified numbers $C_j$ satisfy
$C_j\le B_j$, then
\[
 p(t)\ge\min_{0\le j\le m}C_j
 \qquad(a\le t\le b).
\]
The same conclusion applies piecewise to a finite partition of an interval.
\end{lemma}
\begin{proof}
For $0\le y\le1$, every basis function is nonnegative and
\[
 \sum_{j=0}^{m}\binom mj y^j(1-y)^{m-j}
 =(y+(1-y))^m=1.
\]
The value of the polynomial is therefore a convex combination of its
Bernstein coefficients. In particular,
\[
 p(a+(b-a)y)
 \ge\sum_{j=0}^{m}C_j\binom mj y^j(1-y)^{m-j}
 \ge\min_jC_j.
\]
Applying this on every subinterval of a partition covers the entire
original interval, including its endpoints.

To calculate the Bernstein coefficients, set
$p(a+(b-a)y)=\sum_{k=0}^{m}c_ky^k$, then
\[
 B_j=\sum_{k=0}^{j}c_k\frac{\binom jk}{\binom mk}.
\]
Indeed,
\begin{align*}
 \sum_{j=k}^{m}\frac{\binom jk}{\binom mk}
                   \binom mj y^j(1-y)^{m-j}
 &=y^k\sum_{r=0}^{m-k}\binom{m-k}{r}y^r(1-y)^{m-k-r}\\
 &=y^k
\end{align*}
as desired.  In practice we then only need to check a finite number of inequalities to prove a bound for the entire interval.
\end{proof}

\begin{proof}[Proof of Lemma \ref{lem:input}]
We will prove convergence and the bound
\[
0\leq F(e^{\tau(\beta,\gamma)},e^{-\beta},e^{-\gamma})\leq14
\qquad((\beta,\gamma)\in\mathcal D).
\]
To obtain this uniform bound, we cover $\mathcal D$ by small rectangles and use Lemma \ref{lem:55-rectangle} to reduce each rectangle to its four vertices. For a vertex $(b,g)$ of a rectangle $B$, put $R=e^{\tau(b,g)+E_B}$. Suppose that
\begin{equation}\label{eq:vertex}
 D_5\!\left(Ry,e^{-b},e^{-g}\right)
                 \ge e^{-14}\qquad(0\le y\le1).
\end{equation}
Set $P(t)=D_5(t,e^{-b},e^{-g})$. The affine characters converge absolutely at the vertices used, so the Taylor coefficients of $-\log P$ are the specialized coefficients of $F$ and are nonnegative. If this series had radius $r\leq R$, Pringsheim's theorem \cite[Theorem IV.6]{FS} would make $r$ a singularity. But \eqref{eq:vertex} gives $P>0$ on $[0,R]$, so its logarithm is analytic near $r$. Thus the radius exceeds $R$, and
\[
0\leq F(R,e^{-b},e^{-g})=-\log P(R)\leq14.
\]
Nonnegativity also gives convergence of the full series. Once this holds at all four vertices, Lemma \ref{lem:55-rectangle} gives convergence and $F\leq14$ throughout $B$.

It remains to check \eqref{eq:vertex}. The expression is a polynomial of degree $11$ in $y$, so Lemma \ref{lem:55-bernstein} applies. A computer program bounds its Bernstein coefficients using rigorous estimates of the affine characters, subdividing the interval when needed. Starting from $\mathcal Q$, it subdivides rectangles whose vertex checks fail and discards only rectangles disjoint from $\mathcal D$. This produces a verified covering by $87$ accepted rectangles with $149$ distinct vertices, at which \eqref{eq:vertex} holds. Hence $F$ converges and is at most $14$ throughout $\mathcal D$.

Finally, fix any pair $(q,z)$ in the prescribed family. Convergence at its corresponding radius $T=e^{\tau(\beta,\gamma)}$ makes $F(t,q,z)$ holomorphic on $|t|<T$. Since $D_5=\exp(-F)$ there, $D_5$ has no zero in this disk. At the endpoint, continuity gives
\begin{equation}\label{eq:d5-lower-bound}
D_5(T,q,z)=e^{-F(T,q,z)}\geq e^{-14}.
\end{equation}
Appendix~\ref{app:characters} gives the character formulas and tail
bounds used in these checks. Appendix~\ref{app:high-level-cover}
gives the covering diagram, its verification, and the full table of
$149$ vertices and their certified bounds.
This proves Lemma \ref{lem:input}.
\end{proof}

\section{The finite remainder}\label{sec:remainder}
With both a high depth and a high level bound, the region remaining is
\begin{equation}\label{eq:s6-region}
 \mathcal B=\{(\ell,d,s)\in\Z^3:\;3\le\ell\le999,\quad
 2\ell\le d\le165\ell-1 , \quad
 -\lfloor\ell/2\rfloor\le s\le0\}.
\end{equation}
At each point we must prove
\[
 m(\ell,d,s)\le p\bigl(N(\ell,d,s)\bigr),\qquad
 N(\ell,d,s)=\ell d-\ell^2+1-s^2.
\]
There are $997$ levels and
\[
 |\mathcal B|
 =\sum_{\ell=3}^{999}163\ell\bigl(\lfloor\ell/2\rfloor+1\bigr)
 =27{,}186{,}972{,}935
\]
integer triples.

One could compute all these multiplicities using Peterson's recurrence \cite[Chapter~11]{Kac}.  However, this is still a lot of roots, so the purpose of this section is to reduce the number requiring exact computation.

Set $N_0=180\,000$ and define
\begin{equation}\label{eq:final-remainder}
 \boxed{\begin{aligned}
  \mathcal E&=\{(\ell,d,s)\in\mathcal B:N(\ell,d,s)\le N_0\},\\
  \mathcal R&=\{(\ell,d,s)\in\mathcal B:N(\ell,d,s)>N_0\}.
 \end{aligned}}
\end{equation}

With this split, $|\mathcal E|=25{,}043{,}737$. Exact computation using the Weyl denominator recurrence verifies the bound at every point of $\mathcal E$. We therefore concentrate on proving the bound on $\mathcal R$ using $D_5$.

\subsection{Monotonicity}

The test we need to design is motivated by the following fact.

\begin{lemma}[Monotonicity]
\label{lem:s6-monotonicity}
For integers $\ell\ge1$, $d\ge2\ell$, and
$-\lfloor\ell/2\rfloor\le s\le0$, one has
\begin{align}
 m(\ell,d,s)&\le m(\ell,d,s+1)\qquad(s<0),
                                      \label{eq:s6-charge-monotonicity}\\
 m(\ell,d,s)&\le m(\ell,d+1,s).
                                      \label{eq:s6-depth-monotonicity}
\end{align}
Moreover $m(\ell,d,s)=m(\ell,d,-s)$. These statements are not restricted
to $d<165\ell$ or $\ell<1000$.
\end{lemma}
\begin{proof}
We use a basic injectivity fact. Suppose operators $E,F_0,H$ satisfy
$[E,F_0]=H$, $[H,E]=2E$, and $[H,F_0]=-2F_0$, with $F_0$ locally
nilpotent. Then $E$ is injective on a weight space of weight $\lambda<0$.
Indeed, if $v\ne0$, $Hv=\lambda v$, and $Ev=0$, choose the largest
$r\ge0$ for which $F_0^rv\ne0$. The commutator formula gives
\[
 EF_0^{r+1}v=(r+1)(\lambda-r)F_0^rv.
\]
Its left side is zero, since $F_0^{r+1}v=0$, while its right side is
nonzero: $r+1>0$, $\lambda-r<0$, and $F_0^rv\ne0$. This contradiction
proves injectivity. Local nilpotence of the adjoint simple-root
operators is part of the standard integrability of the adjoint
Kac--Moody action; see \cite{Kac}.

Write $g_{\ell,d,s}$ for the root space with level $\ell$, depth $d$ and charge $s$. The weights for $h_0$ and $h_1$ are
\[
          -\ell-2s\quad\hbox{and}\quad2s,
\]
respectively, by \eqref{eq:pairings}. Adding $\alpha_1$ changes
$(d,s)$ to $(d,s+1)$, whereas adding $\alpha_0$ changes it to
$(d+1,s-1)$. For $s<0$, the weight $2s$ is negative, so
\[
 \operatorname{ad}e_1:
       g_{\ell,d,s}\hookrightarrow g_{\ell,d,s+1}
\]
is injective. This proves \eqref{eq:s6-charge-monotonicity}.

For depth increase with $s<0$, compose
\[
 g_{\ell,d,s}
 \xhookrightarrow{\ \operatorname{ad}e_1\ }
 g_{\ell,d,s+1}
 \xhookrightarrow{\ \operatorname{ad}e_0\ }
 g_{\ell,d+1,s}.
\]
The second map has source weight
\[
 -\ell-2(s+1)=-\ell-2s-2\le-2<0,
\]
since $s\ge-\ell/2$. Both maps are therefore injective. For $s=0$,
reverse their order:
\[
 g_{\ell,d,0}
 \xhookrightarrow{\ \operatorname{ad}e_0\ }
 g_{\ell,d+1,-1}
 \xhookrightarrow{\ \operatorname{ad}e_1\ }
 g_{\ell,d+1,0}.
\]
The source weights are now $-\ell$ and $-2$. This proves
\eqref{eq:s6-depth-monotonicity}, including its boundary case.
Finally, reflection in $\alpha_1$ sends $(\ell,d,s)$ to
$(\ell,d,-s)$ and preserves multiplicities.
\end{proof}

\begin{corollary}
\label{cor:s6-strip}
For $\ell,d,s$ as in Lemma~\ref{lem:s6-monotonicity}, all integers
$r\ge0$ and $k$ with $s\le k\le-s$ satisfy
\begin{equation}\label{eq:s6-strip}
                  m(\ell,d+r,k)\ge m(\ell,d,s).
\end{equation}
\end{corollary}
\begin{proof}
First increase depth $r$ times, obtaining
$m(\ell,d+r,s)\ge m(\ell,d,s)$. Put $k'=-|k|$. The inequalities
$s\le k\le-s$ give $s\le k'\le0$, so repeated charge increase gives
$m(\ell,d+r,k')\ge m(\ell,d+r,s)$. If $k>0$, horizontal reflection
identifies the multiplicities at $k$ and $k'$; otherwise $k=k'$.
\end{proof}

\subsection{The $D_5$ test}

Put
\[
 F(t,q,z)=-\log D_5(t,q,z)
       =\sum_{\ell\ge1}\sum_{d,s\in\Z}c_{\ell,d,s}t^\ell q^dz^s.
\]

Theorem~\ref{thm:odd} gives
\begin{equation}\label{eq:s6-majorant}
                   c_{\ell,d,s}\ge m(\ell,d,s)\ge0.
\end{equation}
For fixed $\ell$, define
\begin{equation}\label{eq:s6-level-series}
 A_\ell (q,z):=[t^\ell]F(t,q,z)
               =\sum_{d,s}c_{\ell,d,s}q^dz^s.
\end{equation}

At positive parameters where the six characters in
\eqref{eq:d5-characters} converge absolutely, $A_\ell$ is finite. To see
this without first assuming convergence in $t$, use
\[
 [t^\ell]F=\sum_{j=1}^{\ell}\frac1j[t^\ell](1-D_5)^j.
\]

\begin{lemma}[Monotonicity-enhanced real bound]\label{lem:s6-real-bound}
Let $\ell\ge1$, $d\ge2\ell$, and
$-\lfloor\ell/2\rfloor\le s\le0$ be integers. Suppose $0<q_0 <1$, $z_0 >0$,
and $A_\ell(q_0,z_0)$ converges. Define
\[
                      S_s(z_0)=\sum_{j=0}^{-2s}z_0^j.
\]
Then
\begin{equation}\label{eq:finite-bound}
 \quad
 m(\ell,d,s)\le
       \frac{(1-q_0)A_\ell(q_0,z_0)}{S_s(z_0)}q_0^{-d}z_0^{-s}.
 \quad
\end{equation}
\end{lemma}
\begin{proof}
For an integer cutoff $J\ge0$, nonnegativity,
\eqref{eq:s6-majorant}, and Corollary~\ref{cor:s6-strip} give
\begin{align*}
 A_\ell(q_0,z_0)
 &\ge\sum_{r=0}^{J}\sum_{k=s}^{-s}
                  c_{\ell,d+r,k}q_0^{d+r}z_0^k\\
 &\ge\sum_{r=0}^{J}\sum_{k=s}^{-s}
                  m(\ell,d+r,k) q_0^{d+r}z_0^k\\
  &\ge\sum_{r=0}^{J}\sum_{k=s}^{-s}
                  m(\ell,d,s) q_0^{d+r}z_0^k\\
 &= m(\ell,d,s)q_0^d
                   \left(\sum_{r=0}^{J}q_0^r\right)
                   \left(\sum_{k=s}^{-s}z_0^k\right)\\
 &=m(\ell,d,s)q_0^d z_0^s\frac{1-q_0^{J+1}}{1-q_0}S_s(z_0).
\end{align*}
Let $J\to\infty$. Since $0<q_0<1$, the factor $q_0^{J+1}$ tends to zero.
Divide by $q_0^d z_0^s S_s(z_0)/(1-q_0)>0$ to obtain the claim.
\end{proof}

Fix a level $3\leq\ell\leq999$ and a nondegenerate rectangle in the ratio region $[2,165]\times[-1/2,0]$. Let $(x_c,u_c)$ be its midpoint. The corresponding integer depth and charge ranges in $\mathcal B$ define a rectangle $R=[a,b]\times[c,e]$, with
\[
2\ell\leq a\leq b\leq165\ell-1,\qquad
-\lfloor\ell/2\rfloor\leq c\leq e\leq0.
\]
Empty cells need no test. The lemma gives an upper bound for any admissible parameters; we choose the familiar point
\begin{equation}\label{eq:s6-midpoint-rule}
 \quad
 \beta=\frac{\pi}{\sqrt{6(x_c-u_c^2-7/8)}},\qquad
 q_0=e^{-\beta},\qquad z_0=q_0^{-2u_c}.
 \quad
\end{equation}
Since $x_c-u_c^2-7/8>0$ and $0<q_0<z_0<1$, the characters converge at this point. Set $\overline A_{\ell,R}=A_\ell(q_0,z_0)$.

\begin{remark}
Write $D_5(t,q,z)=1+\sum_{j=1}^{11}b_jt^j$, with $b_j=0$ for $j>11$. The identity $D_5\partial_tF=-\partial_tD_5$ gives
\[
\ell A_\ell=-\ell b_\ell
 -\sum_{j=1}^{\min(11,\ell-1)}b_j(\ell-j)A_{\ell-j}.
\]
This recurrence is used in the computation. The computed values are approximate. We bound their error, including the error in the character evaluations, and replace the exact value $\overline A_{\ell,R}$ by the resulting upper bound. All the estimates below remain valid with this replacement.
\end{remark}

Set the constant
\begin{equation}\label{eq:s6-rectangle-constant}
       C_R=\frac{(1-q_0)\overline A_{\ell,R}}{S_e(z_0)}.
\end{equation}

If $s\le e$, then $-2s\ge-2e$ and $z_0^j>0$, so
\[
             S_s(z_0)\ge S_e(z_0)>0.
\]

Consequently the same constant works at every integer point of $R$:
\begin{equation}\label{eq:s6-rectangle-bound}
                 m(\ell,d,s)\le C_Rq_0^{-d}z_0^{-s}.
\end{equation}

Next write the logarithm of the partition lower bound
\eqref{eq:partition-bound} as
\begin{equation}\label{eq:s6-partition-log}
 f(n)=\Aconst\sqrt n-\log(4\sqrt3\,n)
                    +\log\left(1-\frac1{2\sqrt n}\right),
       \qquad \Aconst=\pi\sqrt{2/3}.
\end{equation}
For positive integers $n$, we have $e^{f(n)}<p(n)$. For a fixed pair
$(q_0,z_0)$ and the rectangle constant $C_R$ in
\eqref{eq:s6-rectangle-constant}, set
\begin{equation}\label{eq:depth-gap}
        \Phi_R(d,s)=f\bigl(N(\ell,d,s)\bigr)
                         +d\log q_0+s\log z_0-\log C_R.
\end{equation}

\begin{lemma}[Real comparison from four corners]\label{new:lem:concave}
If $\Phi_R$ is strictly positive at all distinct corners of $R$, then
\[
                 m(\ell,d,s)<p\bigl(N(\ell,d,s)\bigr)
                       \qquad((d,s)\in R\cap\Z^2).
\]
\end{lemma}
\begin{proof}
This is a concavity argument. With
$h(n)=\log(1-1/(2\sqrt n))$, direct differentiation gives, for $n\ge3$,
\[
 h'(n)=\frac1{2n(2\sqrt n-1)}>0,\qquad
 h''(n)=-\frac{6\sqrt n-2}{(4n^{3/2}-2n)^2}<0.
\]
Hence
\begin{align*}
 f'(n)&=\frac{\Aconst}{2\sqrt n}-\frac1n+h'(n)
        >\frac{\Aconst\sqrt n-2}{2n}>0,\\
 f''(n)&=-\frac{\Aconst}{4n^{3/2}}+\frac1{n^2}+h''(n)
        <\frac{4-\Aconst\sqrt n}{4n^2}<0.
\end{align*}
The last inequalities use
$\Aconst\sqrt n\ge\Aconst\sqrt3=\pi\sqrt2>4$.
Thus $f$ is increasing and concave on $[3,\infty)$.

On the real chamber slice $d\ge2\ell$, $-\ell/2\le s\le0$,
\[
              N(\ell,d,s)\ge\frac34\ell^2+1>3.
\]
For $v=(d_1,s_1)$, $w=(d_2,s_2)$ in that slice and $0\le\lambda\le1$,
\[
 N(\lambda v+(1-\lambda)w)
   =\lambda N(v)+(1-\lambda)N(w)
                       +\lambda(1-\lambda)(s_1-s_2)^2.
\]
Here $N(v)$ abbreviates $N(\ell,d_1,s_1)$, and similarly for $w$.
First using monotonicity and then concavity of $f$ yields
\begin{align*}
 f\bigl(N(\lambda v+(1-\lambda)w)\bigr)
 &\ge f\bigl(\lambda N(v)+(1-\lambda)N(w)\bigr)\\
 &\ge\lambda f(N(v))+(1-\lambda)f(N(w)).
\end{align*}
The remaining terms of \eqref{eq:depth-gap} are affine in $(d,s)$.
Therefore $\Phi_R$ is concave.

Every point of $R$ is a convex combination of its distinct corners $v_j$, including when $R$ is a line segment or a point. For the corresponding nonnegative weights $\lambda_j$ with $\sum_j\lambda_j=1$, we therefore have
\[
 \Phi_R(d,s)\ge\sum_j\lambda_j\Phi_R(v_j)
                   \ge\min_j\Phi_R(v_j)>0.
\]
At an integer point, \eqref{eq:s6-rectangle-bound} now gives
\[
m(\ell,d,s)\leq C_Rq_0^{-d}z_0^{-s}
 <e^{f(N(\ell,d,s))}<p\bigl(N(\ell,d,s)\bigr),
\]
as desired.
\end{proof}

\subsection{Running the algorithm}
For each level $3\leq\ell\leq999$, start from the full ratio rectangle $[2,165]\times[-1/2,0]$. For each cell, form its integer depth and charge rectangle $R=[a,b]\times[c,e]$ as in the preceding subsection, discarding empty cells. If $N(\ell,b,e)\leq N_0$, then $R\subseteq\mathcal E$, since $N$ increases with depth and charge when $s\leq0$. Such a cell is left to the exact computation. Otherwise apply the $D_5$ test: accept the rectangle if $\Phi_R$ is positive at all its corners, and subdivide the ratio cell dyadically if the test fails.

The computation terminates with $74{,}434$ rectangles passing the $D_5$ test. Their coverage of all $27{,}161{,}929{,}198$ triples in $\mathcal R$ is verified, proving the strict bound throughout that region.

The remaining $25{,}043{,}737$ triples in $\mathcal E$ are checked directly using Peterson's recurrence, and all satisfy the bound. Since $\mathcal B=\mathcal E\sqcup\mathcal R$, this completes the finite verification.

Together with the preceding sections, this proves Theorem \ref{thm:main}.

\newpage

\appendix

\section{A derivation of the level-2 characters}

\begin{proof}[Derivation of \eqref{eq:level2strings}]
We start from the full affine character formula in
\cite[eq. (5.24)]{CKMN}. Write
\[
 T_0(q,z)=\sum_{a\in\Z}q^{2a^2}z^{2a},\qquad
 T_1(q,z)=\sum_{a\in\Z}q^{2a(a+1)}z^{2a+1}.
\]
After reversing roots to lowest-weight, we evaluate
$e^{-2\alpha_{-1}-d\delta-s\alpha_1}$ as $q^dz^s$, suppressing
the common level-two factor. Its fundamental weights are
$\boldsymbol\Lambda_0=-\alpha_{-1}-2\delta$ and
$\boldsymbol\Lambda_1=-\alpha_{-1}-2\delta+\alpha_1/2$
\cite[eq. (2.16)]{CKMN}. Thus the extremal weight
$2\boldsymbol\Lambda_1$ has depth $4$ and charge $-1$, and
\[
 \operatorname{Ch}L(2\boldsymbol\Lambda_1)=q^4\chi_{2,2}(q,z).
\]
The theta functions include the factor
$e^{-(\lambda,\lambda)\delta/4}$ at level two. Since
$(2\boldsymbol\Lambda_0)^2=-8$ and
$(2\boldsymbol\Lambda_1)^2=-6$, their translation sums give
\cite[eqs. (B.5)--(B.6)]{CKMN}
\[
 \begin{aligned}
 \Theta_{2\boldsymbol\Lambda_0}
   &=q^{-2}\sum_{a\in\Z}q^{4+2a^2}z^{2a}
     =q^2T_0(q,z),\\
 \Theta_{2\boldsymbol\Lambda_1}
   &=q^{-3/2}\sum_{a\in\Z}q^{4+2a(a+1)}z^{2a+1}
     =q^{5/2}T_1(q,z).
 \end{aligned}
\]
Put $\varphi(q)=\prod_{a\ge1}(1-q^a)=1/P(q)$. The product identity
$\varphi(q^2)=\varphi(q)Q(q)$ gives
$\varphi(q)/\varphi(q^2)=1/Q(q)$. Substituting the preceding
theta functions into the cited character formula therefore yields
\[
 q^4\chi_{2,2}(q,z)=\frac{q^{3/2}}{2Q(q)}
 \left\{
 q^2\bigl(P(\sqrt q)-P(-\sqrt q)\bigr)T_0(q,z)
 +q^{5/2}\bigl(P(\sqrt q)+P(-\sqrt q)\bigr)T_1(q,z)
 \right\}.
\]
The even and odd powers of $\sqrt q$ in the partition series satisfy
\[
 \frac{P(\sqrt q)+P(-\sqrt q)}2=E(q),\qquad
 \frac{P(\sqrt q)-P(-\sqrt q)}{2\sqrt q}=O(q).
\]
Cancelling $q^4$ now gives the full theta decomposition:
\begin{equation}\label{eq:level2-theta-decomposition}
 \chi_{2,2}(q,z)=\frac{O(q)T_0(q,z)+E(q)T_1(q,z)}{Q(q)}.
\end{equation}
Only $T_0$ contains charge $0$, with coefficient $1$ from $a=0$.
Only $T_1$ contains charge $1$, also with coefficient $1$ from $a=0$.
Extracting these coefficients proves \eqref{eq:level2strings}.
In particular, the constant term in $q$ is $z^{-1}+1+z$.
\end{proof}

\clearpage
\section{Character evaluation and explicit tails}\label{app:characters}
This appendix gives the convergent formulas and explicit remainder
bounds used to evaluate the six affine characters in
\eqref{eq:d5-characters}. Their labels $(k,j)$ are
$(1,0),(2,2),(4,4),(7,6),(10,10),(11,8)$.

\subsection{The normalized affine character formula}
Put $K=k+2$ and $J_*=j+1$. The normalized form of the affine
Weyl--Kac character formula is
\begin{equation}\label{eq:charged-char}
 \chi_{k,j}(q,z)=
 \frac{\displaystyle\sum_{n\in\Z}q^{Kn^2+J_*n}S_{2Kn+J_*}(z)}
 {\displaystyle\prod_{a\ge1}(1-q^a)(1-zq^a)(1-q^a/z)},
 \qquad
 S_b(z)=\frac{z^{b/2}-z^{-b/2}}{z^{1/2}-z^{-1/2}}.
\end{equation}
This is the ordinary affine character theorem
\cite[Chapter 10]{Kac}, with the horizontal Weyl denominator divided
out. The level of the shifted highest
weight is $k+2$ and its shifted horizontal label is $j+1$.
The two horizontal Weyl terms in each affine translation orbit form
$S_{2Kn+J_*}$; the translation contributes $Kn^2+J_*n$ to depth.
The $n=0$ term is $S_{j+1}(z)=\sum_{a=-j/2}^{j/2}z^a$,
with the sum taken in steps of one. The three product
factors record the positive imaginary root and the two positive real
root strings after removing the zero-mode horizontal root.

For every nonzero integer $b$,
\begin{equation}\label{eq:S-polynomial}
 S_b(z)=\operatorname{sgn}(b)
       \sum_{a=-(|b|-1)/2}^{(|b|-1)/2}z^a.
\end{equation}
The following estimates establish convergence of the numerator
and bound every omitted term.

\subsection{Numerical estimate}
In this section we describe how the characters are estimated numerically.

For $0<q<z\le1$, truncate the numerator
at $|n|<B$. The Laurent-polynomial expression above gives
\[
 |S_{2Kn+J_*}(z)|\le(2K|n|+J_*)
                    q^{-K|n|-(J_*-1)/2}.
\]
For $|n|=B+h$, the lower bound on the exponent increases by at least
$(2KB-J_*)h$. Summing the resulting geometric series and its first
moment bounds the total omitted numerator by
\begin{equation}\label{eq:charged-tail}
 T_{k,j}(B;q)=
 2q^{KB^2-(K+J_*)B-(J_*-1)/2}
 \left(\frac{2KB+J_*}{1-\sigma}+\frac{2K\sigma}{(1-\sigma)^2}\right),
 \quad\sigma=q^{2KB-J_*}<1.
\end{equation}
The factor two includes both signs of $n$. The evaluator uses $B=20$.

After retaining denominator factors through $H$, put
\[
 S=\frac{q^{H+1}(1+z+z^{-1})}{1-q}.
\]
The omitted factors have the form $1-a_j$ with $0\le a_j<1$
and $\sum_j a_j=S$. When $S<1$, their product belongs to $[1-S,1]$,
by the inequality
\begin{equation}\label{eq:product-tail}
 1-\sum_j a_j\le\prod_j(1-a_j)\le1
 \qquad(0\le a_j<1,\ \sum_j a_j<1)
\end{equation}
which follows by finite induction and monotone limits.
In numerical evaluation, the product cutoff is
$H=\max(60,\lceil50/\beta_-\rceil)$, where $\beta_->0$ is the
lower endpoint of the interval enclosing $\beta$.
\clearpage
\section{The finite covering for the large-level bound}\label{app:high-level-cover}
This appendix records details on the finite covering used in the proof of
Lemma~\ref{lem:input}. It covers the continuous parameter domain for
the large-level argument in Section~\ref{sec:homology}.

\subsection{The rectangles and their coverage}
The domain $\mathcal D$ in \eqref{eq:real-domain} is contained in
$\mathcal Q=[1/10,11/8]\times[0,11/8]$.
We construct an exact
rational subdivision of $\mathcal Q$ into $87$ accepted and $60$
excluded rectangles, with $149$ distinct vertices belonging to
accepted rectangles. Figure~\ref{fig:high-level-cover} shows this
subdivision, including the fine subdivision near the upper-right
boundary of $\mathcal D$.

Coverage is checked by an exact rational sweep. Between each pair of
consecutive $\beta$-coordinates of rectangle edges, the corresponding
$\gamma$-intervals concatenate from $0$ to $11/8$, with no gaps or
overlapping interiors. An excluded rectangle
$[b_0,b_1]\times[g_0,g_1]$ must satisfy at least one of
\[
 g_0>b_1,\qquad
 \frac{\pi^2}{6}+\frac{g_1^2}{4}<\frac98b_0^2,\qquad
 \frac{\pi^2}{6}+\frac{g_0^2}{4}>\frac{1313}{8}b_1^2.
\]
Each test contradicts a defining inequality of $\mathcal D$ throughout
the closed rectangle. The strict comparisons preserve boundary points;
the accepted rectangles therefore cover all of $\mathcal D$.

\begin{figure}[htbp]
\centering
\includegraphics[width=\linewidth]{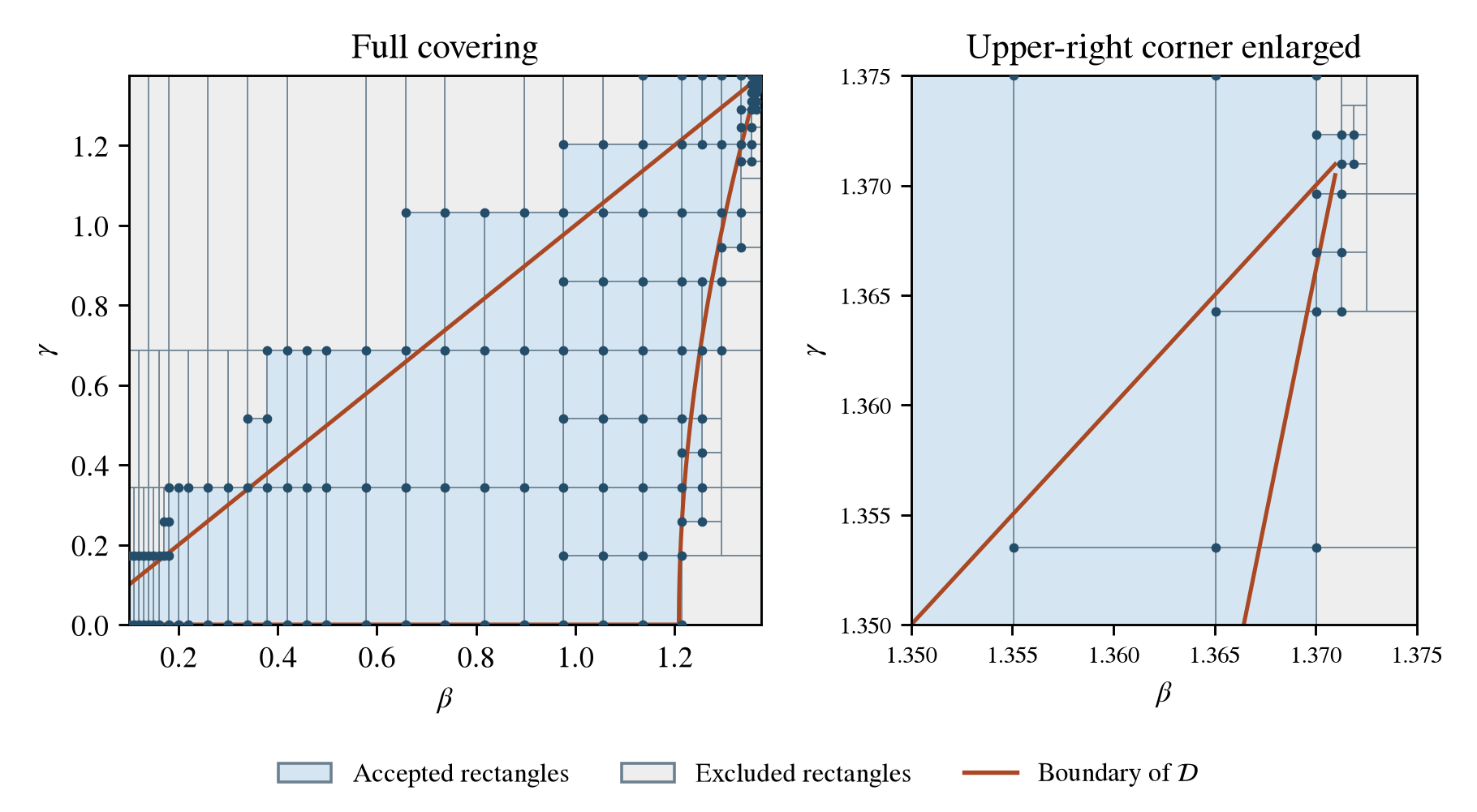}
\caption{The actual rational covering used for the large-level bound.
Blue rectangles are accepted and gray rectangles are excluded.
Dots mark the $149$ accepted vertices; the orange curves mark the
boundary of $\mathcal D$. The right panel enlarges the fine subdivision
near $(11/8,11/8)$.}
\label{fig:high-level-cover}
\end{figure}

\subsection{What is checked at each vertex}
For a shared vertex $v=(b,g)$, the checker computes an outward upper
enclosure $\widehat E_v$ of the largest allowance $E_B$ from
\eqref{eq:error} among all accepted rectangles incident to $v$.
It then sets
\[
 R_v=\exp\bigl(\tau(b,g)+\widehat E_v\bigr),\qquad
 P_v(y)=D_5(R_vy,e^{-b},e^{-g}).
\]
The six characters in \eqref{eq:d5-characters} are enclosed using
Appendix~\ref{app:characters}. Every vertex satisfies $0\le g<2b$,
so the folding procedure there applies. Computations were done using
interval arithmetic at $60$ decimal digits.

Write $P_v(y)=\sum_{k=0}^{11}a_ky^k$. Its degree-eleven Bernstein
coefficients on $[0,1]$ are
\[
 c_j=\sum_{k=0}^j a_k\frac{\binom jk}{\binom{11}k},
 \qquad 0\le j\le11.
\]
If their lower endpoints do not all exceed $e^{-14}$, the checker
bisects the interval and transforms the coefficients by repeated
midpoint averaging (the de Casteljau algorithm). At most three
successive bisections suffice at every vertex. The smallest lower
endpoint over all accepted subintervals is a lower bound for $P_v$
on the entire interval, by Lemma~\ref{lem:55-bernstein}.

Table~\ref{tab:high-level-vertices} lists every vertex and its lower
bound $L_v$, rounded down to six significant digits. In particular,
\[
 P_v(y)\ge L_v\ge9.23354\times10^{-7}>e^{-14}
 \qquad(0\le y\le1).
\]
The convergence argument in Lemma~\ref{lem:input} then gives $F\le14$
at the corrected vertices, and Lemma~\ref{lem:55-rectangle} propagates
this bound throughout the covering.

\clearpage
\begingroup
\small
\setlength{\tabcolsep}{12pt}
\begin{longtable}{r r r l}
\caption{All 149 accepted vertices, ordered by $b$ and then $g$.
Coordinates are exact; $L_v$ is a downward-rounded certified lower bound for $P_v$ on $[0,1]$.}\label{tab:high-level-vertices}\\
\toprule No. & $b$ & $g$ & $L_v$\\ \midrule
\endfirsthead
\multicolumn{4}{c}{\tablename\ \thetable\ (continued)}\\
\toprule No. & $b$ & $g$ & $L_v$\\ \midrule
\endhead
\midrule\multicolumn{4}{r}{Continued on next page}\\\endfoot
\bottomrule\endlastfoot
1 & $1/10$ & $0$ & $1.54390\times10^{-1}$\\
2 & $1/10$ & $11/64$ & $1.54390\times10^{-1}$\\
3 & $563/5120$ & $0$ & $1.46485\times10^{-1}$\\
4 & $563/5120$ & $11/64$ & $1.46485\times10^{-1}$\\
5 & $307/2560$ & $0$ & $1.48765\times10^{-1}$\\
6 & $307/2560$ & $11/64$ & $1.48765\times10^{-1}$\\
7 & $133/1024$ & $0$ & $1.48331\times10^{-1}$\\
8 & $133/1024$ & $11/64$ & $1.48331\times10^{-1}$\\
9 & $179/1280$ & $0$ & $1.46203\times10^{-1}$\\
10 & $179/1280$ & $11/64$ & $1.46203\times10^{-1}$\\
11 & $767/5120$ & $0$ & $1.43002\times10^{-1}$\\
12 & $767/5120$ & $11/64$ & $1.43002\times10^{-1}$\\
13 & $409/2560$ & $0$ & $1.15464\times10^{-1}$\\
14 & $409/2560$ & $11/64$ & $1.15464\times10^{-1}$\\
15 & $869/5120$ & $11/64$ & $1.43776\times10^{-1}$\\
16 & $869/5120$ & $33/128$ & $1.43776\times10^{-1}$\\
17 & $23/128$ & $0$ & $8.56728\times10^{-2}$\\
18 & $23/128$ & $11/64$ & $1.01824\times10^{-1}$\\
19 & $23/128$ & $33/128$ & $1.37271\times10^{-1}$\\
20 & $23/128$ & $11/32$ & $8.56728\times10^{-2}$\\
21 & $511/2560$ & $0$ & $7.29104\times10^{-2}$\\
22 & $511/2560$ & $11/32$ & $7.29104\times10^{-2}$\\
23 & $281/1280$ & $0$ & $3.63779\times10^{-2}$\\
24 & $281/1280$ & $11/32$ & $3.63779\times10^{-2}$\\
25 & $83/320$ & $0$ & $1.62364\times10^{-2}$\\
26 & $83/320$ & $11/32$ & $1.62364\times10^{-2}$\\
27 & $383/1280$ & $0$ & $2.79246\times10^{-2}$\\
28 & $383/1280$ & $11/32$ & $2.79246\times10^{-2}$\\
29 & $217/640$ & $0$ & $3.02749\times10^{-2}$\\
30 & $217/640$ & $11/32$ & $3.02749\times10^{-2}$\\
31 & $217/640$ & $33/64$ & $5.36187\times10^{-2}$\\
32 & $97/256$ & $0$ & $2.89552\times10^{-2}$\\
33 & $97/256$ & $11/32$ & $2.89552\times10^{-2}$\\
34 & $97/256$ & $33/64$ & $4.25988\times10^{-2}$\\
35 & $97/256$ & $11/16$ & $3.45112\times10^{-2}$\\
36 & $67/160$ & $0$ & $2.62294\times10^{-2}$\\
37 & $67/160$ & $11/32$ & $2.56916\times10^{-2}$\\
38 & $67/160$ & $11/16$ & $2.56916\times10^{-2}$\\
39 & $587/1280$ & $0$ & $2.30739\times10^{-2}$\\
40 & $587/1280$ & $11/32$ & $2.26902\times10^{-2}$\\
41 & $587/1280$ & $11/16$ & $2.26902\times10^{-2}$\\
42 & $319/640$ & $0$ & $9.68846\times10^{-3}$\\
43 & $319/640$ & $11/32$ & $8.81053\times10^{-3}$\\
44 & $319/640$ & $11/16$ & $8.81053\times10^{-3}$\\
45 & $37/64$ & $0$ & $1.54325\times10^{-3}$\\
46 & $37/64$ & $11/32$ & $7.50554\times10^{-4}$\\
47 & $37/64$ & $11/16$ & $7.50553\times10^{-4}$\\
48 & $421/640$ & $0$ & $2.57386\times10^{-3}$\\
49 & $421/640$ & $11/32$ & $2.12639\times10^{-3}$\\
50 & $421/640$ & $11/16$ & $2.12647\times10^{-3}$\\
51 & $421/640$ & $33/32$ & $5.43234\times10^{-3}$\\
52 & $59/80$ & $0$ & $2.16737\times10^{-3}$\\
53 & $59/80$ & $11/32$ & $1.90618\times10^{-3}$\\
54 & $59/80$ & $11/16$ & $1.47064\times10^{-3}$\\
55 & $59/80$ & $33/32$ & $1.47043\times10^{-3}$\\
56 & $523/640$ & $0$ & $1.40754\times10^{-3}$\\
57 & $523/640$ & $11/32$ & $1.25324\times10^{-3}$\\
58 & $523/640$ & $11/16$ & $9.95576\times10^{-4}$\\
59 & $523/640$ & $33/32$ & $9.95403\times10^{-4}$\\
60 & $287/320$ & $0$ & $6.99270\times10^{-4}$\\
61 & $287/320$ & $11/32$ & $6.08725\times10^{-4}$\\
62 & $287/320$ & $11/16$ & $4.57079\times10^{-4}$\\
63 & $287/320$ & $33/32$ & $4.57273\times10^{-4}$\\
64 & $125/128$ & $0$ & $1.82161\times10^{-4}$\\
65 & $125/128$ & $11/64$ & $2.14981\times10^{-3}$\\
66 & $125/128$ & $11/32$ & $1.29868\times10^{-4}$\\
67 & $125/128$ & $33/64$ & $2.12732\times10^{-3}$\\
68 & $125/128$ & $11/16$ & $4.14615\times10^{-5}$\\
69 & $125/128$ & $55/64$ & $2.07376\times10^{-3}$\\
70 & $125/128$ & $33/32$ & $4.14929\times10^{-5}$\\
71 & $125/128$ & $77/64$ & $2.02413\times10^{-3}$\\
72 & $169/160$ & $0$ & $9.82137\times10^{-4}$\\
73 & $169/160$ & $11/64$ & $9.78393\times10^{-4}$\\
74 & $169/160$ & $11/32$ & $9.75856\times10^{-4}$\\
75 & $169/160$ & $33/64$ & $9.71680\times10^{-4}$\\
76 & $169/160$ & $11/16$ & $9.62171\times10^{-4}$\\
77 & $169/160$ & $55/64$ & $9.43866\times10^{-4}$\\
78 & $169/160$ & $33/32$ & $9.14514\times10^{-4}$\\
79 & $169/160$ & $77/64$ & $9.12256\times10^{-4}$\\
80 & $727/640$ & $0$ & $4.00955\times10^{-4}$\\
81 & $727/640$ & $11/64$ & $4.01423\times10^{-4}$\\
82 & $727/640$ & $11/32$ & $4.06375\times10^{-4}$\\
83 & $727/640$ & $33/64$ & $4.12914\times10^{-4}$\\
84 & $727/640$ & $11/16$ & $4.17223\times10^{-4}$\\
85 & $727/640$ & $55/64$ & $4.15597\times10^{-4}$\\
86 & $727/640$ & $33/32$ & $4.05388\times10^{-4}$\\
87 & $727/640$ & $77/64$ & $4.06132\times10^{-4}$\\
88 & $727/640$ & $11/8$ & $4.72513\times10^{-4}$\\
89 & $389/320$ & $0$ & $8.71033\times10^{-5}$\\
90 & $389/320$ & $11/64$ & $8.93929\times10^{-5}$\\
91 & $389/320$ & $33/128$ & $2.28792\times10^{-4}$\\
92 & $389/320$ & $11/32$ & $9.67473\times10^{-5}$\\
93 & $389/320$ & $55/128$ & $2.56797\times10^{-4}$\\
94 & $389/320$ & $33/64$ & $1.06718\times10^{-4}$\\
95 & $389/320$ & $11/16$ & $1.16118\times10^{-4}$\\
96 & $389/320$ & $55/64$ & $1.21874\times10^{-4}$\\
97 & $389/320$ & $33/32$ & $1.21788\times10^{-4}$\\
98 & $389/320$ & $77/64$ & $1.15007\times10^{-4}$\\
99 & $389/320$ & $11/8$ & $1.12961\times10^{-4}$\\
100 & $1607/1280$ & $33/128$ & $6.11498\times10^{-5}$\\
101 & $1607/1280$ & $11/32$ & $7.35851\times10^{-5}$\\
102 & $1607/1280$ & $55/128$ & $8.84794\times10^{-5}$\\
103 & $1607/1280$ & $33/64$ & $7.07627\times10^{-5}$\\
104 & $1607/1280$ & $11/16$ & $9.45703\times10^{-5}$\\
105 & $1607/1280$ & $55/64$ & $1.17007\times10^{-4}$\\
106 & $1607/1280$ & $77/64$ & $1.37936\times10^{-4}$\\
107 & $1607/1280$ & $11/8$ & $1.36007\times10^{-4}$\\
108 & $829/640$ & $11/16$ & $1.66137\times10^{-5}$\\
109 & $829/640$ & $55/64$ & $2.42633\times10^{-6}$\\
110 & $829/640$ & $121/128$ & $7.66565\times10^{-5}$\\
111 & $829/640$ & $33/32$ & $1.64878\times10^{-6}$\\
112 & $829/640$ & $77/64$ & $2.25209\times10^{-6}$\\
113 & $829/640$ & $11/8$ & $5.56136\times10^{-5}$\\
114 & $1709/1280$ & $121/128$ & $6.94480\times10^{-6}$\\
115 & $1709/1280$ & $33/32$ & $9.23354\times10^{-7}$\\
116 & $1709/1280$ & $297/256$ & $3.84792\times10^{-5}$\\
117 & $1709/1280$ & $77/64$ & $5.76114\times10^{-6}$\\
118 & $1709/1280$ & $319/256$ & $4.36511\times10^{-5}$\\
119 & $1709/1280$ & $165/128$ & $3.88335\times10^{-5}$\\
120 & $1709/1280$ & $11/8$ & $6.87295\times10^{-6}$\\
121 & $3469/2560$ & $297/256$ & $1.02720\times10^{-5}$\\
122 & $3469/2560$ & $77/64$ & $1.32461\times10^{-5}$\\
123 & $3469/2560$ & $319/256$ & $1.54968\times10^{-5}$\\
124 & $3469/2560$ & $165/128$ & $1.26364\times10^{-5}$\\
125 & $3469/2560$ & $671/512$ & $1.97403\times10^{-5}$\\
126 & $3469/2560$ & $341/256$ & $2.00549\times10^{-5}$\\
127 & $3469/2560$ & $693/512$ & $2.01679\times10^{-5}$\\
128 & $3469/2560$ & $11/8$ & $1.33356\times10^{-5}$\\
129 & $6989/5120$ & $165/128$ & $7.12583\times10^{-6}$\\
130 & $6989/5120$ & $671/512$ & $7.69019\times10^{-6}$\\
131 & $6989/5120$ & $341/256$ & $8.06851\times10^{-6}$\\
132 & $6989/5120$ & $1375/1024$ & $8.64780\times10^{-6}$\\
133 & $6989/5120$ & $693/512$ & $8.26016\times10^{-6}$\\
134 & $6989/5120$ & $1397/1024$ & $8.74933\times10^{-6}$\\
135 & $6989/5120$ & $11/8$ & $8.26696\times10^{-6}$\\
136 & $14029/10240$ & $1375/1024$ & $3.11431\times10^{-6}$\\
137 & $14029/10240$ & $693/512$ & $3.20692\times10^{-6}$\\
138 & $14029/10240$ & $1397/1024$ & $3.25395\times10^{-6}$\\
139 & $14029/10240$ & $5599/4096$ & $3.38739\times10^{-6}$\\
140 & $14029/10240$ & $2805/2048$ & $3.38924\times10^{-6}$\\
141 & $14029/10240$ & $5621/4096$ & $3.38822\times10^{-6}$\\
142 & $14029/10240$ & $11/8$ & $3.25561\times10^{-6}$\\
143 & $56167/40960$ & $1397/1024$ & $2.05176\times10^{-6}$\\
144 & $56167/40960$ & $5599/4096$ & $2.05774\times10^{-6}$\\
145 & $56167/40960$ & $2805/2048$ & $2.06088\times10^{-6}$\\
146 & $56167/40960$ & $11231/8192$ & $2.06762\times10^{-6}$\\
147 & $56167/40960$ & $5621/4096$ & $2.06119\times10^{-6}$\\
148 & $22477/16384$ & $11231/8192$ & $1.41033\times10^{-6}$\\
149 & $22477/16384$ & $5621/4096$ & $1.41046\times10^{-6}$\\
\end{longtable}
\endgroup
\clearpage
\phantomsection
\end{document}